\documentclass{article}
\usepackage[utf8]{inputenc}
\usepackage{amsfonts,amsmath}
\usepackage{mdframed}
\usepackage{graphicx}
\usepackage{hyperref}

\usepackage{multirow}
\usepackage{array}
\usepackage{subcaption}

\usepackage{tikz}
\usetikzlibrary{arrows.meta,,positioning}

\usepackage{caption}

\usepackage[ruled,boxed,linesnumbered,commentsnumbered,longend]{algorithm2e}

\usepackage{amsthm}
\DeclareMathOperator{\sign}{sign}

\usepackage{multirow}

\newtheorem{proposition}{Proposition}

\newtheorem{assumption}{Assumption}[section]
\newtheorem{example}{Example}[section]

\def\old#1{}

\def\re{\Re}
\def\rn{\Re^n}

\def\m{\mu}

\def\frac#1#2{{#1\over #2}}

\old{

\newtheorem{example}{\hspace{2em }Example\;}[section]

\newtheorem{assumption}{\hspace{2em }Assumption\;}[section]

}

\title{Semilinear Dynamic Programming: Analysis, Algorithms, and Certainty Equivalence Properties\footnote{
This work was carried out at the Fulton School of Computing, and Augmented Intelligence, Arizona State University, Tempe, AZ.}}

\author{Yuchao Li and Dimitri Bertsekas}

\usepackage{setspace}
\begin{document}
\maketitle

\begin{abstract}
We consider a broad class of dynamic programming (DP) problems that involve a partially linear structure and some positivity properties in their system equation and cost function. We address deterministic and stochastic problems, possibly with Markov jump parameters. We focus primarily on infinite horizon problems and prove that under our assumptions, the optimal cost function is linear, and that an optimal policy can be computed efficiently with standard DP algorithms. Moreover, we show that forms of certainty equivalence hold for our stochastic problems, in analogy with the classical linear quadratic optimal control problems.
\end{abstract}


\vspace{1pc}
\section{Introduction}\label{sec:intro}
There are quite a few problems in dynamic programming (DP for short), which are structured favorably in the sense that they possess properties that facilitate their analysis and computational solution. Examples of such properties are convexity or piecewise linearity of the optimal cost function,  a graph or grid problem structure, and the existence of special types of optimal policies in application contexts such as scheduling, inventory control, dynamic portfolio selection, and others. Perhaps the most prominent class of favorably structured problems arise in linear-quadratic optimal control  (linear system equation and quadratic cost function), where the optimal cost function is quadratic and the optimal policy is a linear function of the state. Linear-quadratic problems are also remarkable in that the solution of their stochastic versions has a certainty equivalence property: the optimal policy for a stochastic version of the problem is the same as the one obtained from a deterministic problem after the uncertain quantities have been replaced by their expected values.

A common characteristic of favorably structured DP problems is that they involve {\it special classes of cost functions $\widehat{\cal J}$ and policies  $\widehat {\cal M}$ that are closed under value and policy iteration\/}. By this we mean that the value iteration (VI) algorithm (i.e., the repeated applications of DP iterations), when started with a function $J\in \widehat{\cal J}$, generates a sequence of functions in $\widehat{\cal J}$. Moreover, the policy iteration (PI) algorithm starting from a policy in  $\widehat {\cal M}$, generates policies in $\widehat {\cal M}$.  In addition, every policy in $\widehat {\cal M}$ has a cost function that belongs to $\widehat{\cal J}$.  As a consequence of these relations, $\widehat{\cal J}$ and  $\widehat {\cal M}$ form an interconnected structured pair that lies at the heart of the methodology of favorably structured DP problems.

In this paper we consider a class of DP problems, which involve an $n$-dimensional system with a partially linear structure, and have properties that are qualitatively comparable to those of linear-quadratic problems. In particular, under our assumptions, we show that:
\begin{itemize}
\item[(a)] For deterministic problems (Sections \ref{sec:dete_formulation}-\ref{sec:algo_dete}), the optimal cost function is a linear function of the state (cf.\ the set of cost functions $\widehat{\cal J}$). Moreover, an optimal policy exists within a corresponding specially structured set  (cf.\ the set of policies $\widehat {\cal M}$), and can be efficiently computed with standard DP algorithmic methodology.
\item[(b)] For problems involving stochastic parameters, which are independent over time (Section \ref{sec:stochastic}), a classical form of the certainty equivalence principle holds. 
\item[(c)] For problems involving Markov jump parameters, which evolve over time according to a Markov chain (Section \ref{sec:markov}), a somewhat different type of certainty equivalence holds. In particular, there is a deterministic problem with favorable structure, which is obtained by replacing various stochastic quantities of the Markov jump problem with their expected values. The optimal policies and cost functions of the two problems are closely related, so that the Markov jump problem can be solved with the deterministic DP algorithmic methodology developed in Section \ref{sec:algo_dete}.
\end{itemize}

We note that some results of this type are known for finite horizon problems. The following example (originally given in \cite[Exercises 13-15, p.~67]{bertsekas1976dynamic}, and referred to as \emph{semilinear DP}) illustrates the key structure underlying our analysis—namely, the closure of a special class of cost functions under value iteration and the associated certainty equivalence property. Nevertheless, this finite-horizon problem is far simpler and structurally somewhat different than the infinite horizon problems that we address (it does not have the positivity structure).

\begin{example}
[A Finite Horizon Semilinear Problem] 
\label{examplefinitehorizon}
Consider a problem involving the system
$$x_{k+1} = A_kx_k + f_k(u_k) + w_k,\quad k=0,\ldots,N-1,$$
where $N$ is the length of the control horizon, the state $x_k$ is a vector in $\rn$, the control $u_k$ belongs to a set $U_k$, $f_k$ are given functions, and $A_k$ and $w_k$ are random $n\times n$ matrices and $n$-dimensional vectors, respectively, with given probability distributions that do not depend on $x_k$, $u_k$, or prior values of $A_k$ and
$w_k$.  Let also the cost function be linear in the state and have the form
$$E \left\{q_N'x_N + \sum_{k=0}^{N-1}
\big(q_k'x_k + g_k(u_k)\big)\right\},$$
where the expected value $E\{\cdot\}$ above is taken with respect to the distribution of $A_k$ and $w_k$, $q_k$ are given vectors  in $\rn$, $g_k$ are given functions, and a prime denotes transposition, here and later.  Then, assuming that the
optimal cost for this problem is finite,  it can be shown by induction that the cost-to-go functions of the DP algorithm are affine (linear plus constant). In particular, the DP algorithm generates the optimal cost-to-go functions $J_k(x_k)$ from states $x_k$ according to
$$J_k(x_k)=\min_{u_k\in U_k}\Big[q_k'x_k + g_k(u_k)+E\big\{J_{k+1}(A_kx_k+ f_k(u_k) + w_k)\big\}\Big],$$
starting with  the terminal cost function
$$J_N(x_N)=q_N'x_N.$$
Assuming that $J_{k+1}(x_{k+1})$ is linear-plus-constant of the form
$$J_{k+1}(x_{k+1})=c_{k+1}'x_{k+1}+d_{k+1},$$
where $c_{k+1}$ is a vector in $\re^n$ and $d_{k+1}$ is a scalar, it can be seen with a straighforward calculation that 
$$J_{k}(x_{k})=c_{k}'x_{k}+d_{k},$$
where
$$c_k=q_k+{E\{A_{k}\}'c_{k+1}},$$
$$d_k=\min_{u_k\in U_k}\Big[g_k(u_k)+c_{k+1}'f_k(u_k)\Big]+{c'_{k+1}E\{w_{k}\}}+d_{k+1}.$$
This confirms that the class of linear-plus-constant cost functions is closed under value iteration (cf.\ the class of functions $\widehat{\cal J}$), and  has the structure that underlies the methodology of this paper. Moreover, certainty equivalence holds in the sense that the optimal policy is the same as for the deterministic problem, where the random quantities $A_k$ and $w_k$ are replaced with their expected values $E\{A_{k}\}$ and $E\{w_{k}\}$. Here the optimal policy is obtained from the preceding minimization, and is independent of the initial state and the generated state sequence (cf.\ the class of policies  $\widehat {\cal M}$).
\end{example}

The purpose of this paper is to present a new analysis of a class of infinite horizon semilinear problems, which is qualitatively similar but far more challenging than the one for the preceding finite horizon problem. We focus primarily on the case of stationary $n$-dimensional positive semilinear systems that involve a nonnegative matrix $A$ and cost vector $q$, and conditions that ensure that the state $x_k$ is confined to the positive orthant of $\re^n$. These assumptions (to be spelled out more precisely in subsequent sections) bring to bear the theory of monotone increasing and affine monotonic problems of infinite horizon DP \cite{bertsekas1975monotone,bertsekas1977monotone,bertsekas2019affine}. Among others, our results relate to the methodology of positive linear systems that has been pioneered by Rantzer and his co-workers \cite{rantzer2022explicit,li2024exact,ohlin2023optimal,ohlin2024heuristic,bencherki2024data}. They are also related to the optimal control of continuous-time compartmental systems studied by Blanchini et al. \cite{blanchini2023optimal}. Our emphasis here is different: we develop a discrete-time semilinear DP framework, together with stochastic and Markov jump extensions and associated certainty equivalence results.

The paper is organized as follows. In Section~\ref{sec:dete_formulation}, we formulate a deterministic semilinear DP problem, which will be the starting point for the subsequent analysis and extensions. In Section~\ref{sec:bellman_dete}, we study the solution properties of the corresponding Bellman's equation. In Section~\ref{sec:algo_dete}, we provide computational approaches to solve the deterministic problem. In Section~\ref{sec:stochastic}, we introduce a stochastic extension that involves multiplicative stochastic parameters, which are independent across stages, and we demonstrate the certainty equivalence principle for this problem. In Section~\ref{sec:markov}, we study the case where the stochastic parameters of different stages evolve according to a Markov chain.  

\subsection*{Notation}
We denote by $\Re$ the real line and by $\Re^n$ the set of $n$-dimensional vectors. The set of vectors of $\Re^n$ that have nonnegative components (the positive orthant) is denoted by $\Re^n_+$. All vectors are meant to be column vectors, and a prime denotes transposition, so the inner product of two vectors $x$ and $y$ in $\Re^n$ is denoted by $x'y$. We use $E\{\cdot\}$ to denote expected value. The random quantities with respect to which the expectation is taken will be either listed below the symbol $E$ or will be clear from the context. We use similar notation for conditional expected value. All inequalities involving vectors and  functions are meant to be pointwise. In particular, for a vector $x$ we write $x\ge 0$ (or $x>0$) if all the components of $x$ are nonnegative (strictly positive, respectively). Moreover, for any two functions $J,\hat{J}:X\mapsto\re$, where $X$ is some set, we write $J\ge \hat{J}$ if $J(x)\ge \hat{J}(x)$ for all $x\in X$.

\section{Deterministic Positive Semilinear Problems Over an Infinite Horizon}\label{sec:dete_formulation}

In this section, we will introduce an infinite horizon deterministic stationary semilinear DP problem, which will also be the starting point of our analysis in subsequent sections. Here the state space is a subset $X$ of $\re^n_+$, the control space is denoted by $U$, and the control is constrained to lie in a given nonempty subset $U(x)\subset U$ that may depend on $x$. 
Given some $x_0\in X$, our  problem is 
\begin{equation}
    \label{eq:proq_dete}
    \begin{aligned}
        \min_{\{u_k\}_{k=0}^{\infty}}& \quad \sum_{k=0}^{\infty} \alpha^k g(x_k,u_k)\\
	\mathrm{s.\,t.} & \quad x_{k+1}=f(x_k,u_k),\quad k = 0,1,\ldots,\\
	& \quad u_k\in U(x_k),\quad k = 0,1,\ldots,
    \end{aligned}
\end{equation}
where $f:X\times U\mapsto \re^n$ and $g:X\times U\mapsto \re$ are the system function and cost per stage, respectively, and $\alpha\in (0,1]$ is a given scalar. In particular, we require that $f(x,u)\in X$ for all $x$ and $u\in U(x)$.

The preceding problem formulation is very general and in this paper we treat only some special cases of this problem. Some important examples are discrete-time positive bilinear systems closely related to the continuous-time compartmental control problems studied in \cite{blanchini2023optimal}, a special form of switching system, a type of Markov decision problem with positive stage costs, and positive linear systems with constraints coupled over control components \cite{li2024exact}. We describe some of these problems in the subsequent examples. It is convenient to describe the system equation $f$ and stage cost $g$ in general terms in view of the variety of the system equations and stage costs addressed by our methodology. Moreover, the problem formulation \eqref{eq:proq_dete} also extends naturally to the stochastic and Markov jump problems that we study in Sections~\ref{sec:stochastic} and \ref{sec:markov}.

A key assumption is that the cost per stage $g$ is nonnegative:
$$g(x,u)\ge0,\qquad \hbox{for all } x\in X, u\in U(x).$$
As a result, the problem can be analyzed using the general theory of nonnegative cost DP problems, which among others asserts that the search for an optimal policy can be confined to stationary policies, i.e., functions $\m$ from states to controls, such that $\m(x)\in U(x)$ for all $x\in X$ (see Appendix~\ref{app:a}).

Consistent with the discussion of the preceding section, we will focus on semilinear-type problems, involving a structured set of cost functions $\widehat{\cal J}$ and a corresponding subset of stationary policies $\widehat {\cal M}$. In particular, $\widehat{\cal J}$ consists of nonnegative linear functions $J(x)=c'x$, where $c\ge0$, and $\widehat {\cal M}$ consists of policies $\m$ for which there exists an $n\times n$ nonnegative matrix $A_\m$, and $n$-dimensional vectors $q_\m$ such that
\begin{equation}
    \label{eq:semil_cond}
    f\big(x,\m(x)\big)=A_\m x,\qquad g\big(x,\m(x)\big)=q_\m'x.
\end{equation}
Several interesting examples where this structure arises will be given in what follows.

We make the following standing assumptions, which will hold throughout 
Sections~\ref{sec:dete_formulation}-\ref{sec:algo_dete}. These assumptions define the class of problems to which our framework applies. The subsequent examples illustrate important cases where it can be verified.

\begin{assumption}\label{asm:dete_well}
\begin{itemize}
    \item[(a)] \emph{Closure and Attainability}: The set of nonnegative linear functions $\widehat{\cal J}$ is closed under VI in the sense that 
for every $c\in\re^n_+$, the function
$$\min_{u\in U(x)}\Big[g(x,u)+\alpha c'f(x,u)\Big]$$
belongs to  $\widehat{\cal J}$, i.e., it has the form $\hat c'x$ for some unique $\hat c\ge0$. Furthermore, $\hat c$ depends continuously on $c$. Moreover, there is a policy $\m\in\widehat {\cal M}$ that attains the minimum above, in the sense that
\begin{equation}
    \label{eq:common_policy_min}
    \m(x)\in\arg\min_{u\in U(x)}\Big[g(x,u)+\alpha c'f(x,u)\Big],\qquad \text{for all }x\in X.
\end{equation}
\item[(b)] \emph{Stabilizability}: There exists a policy $\hat{\mu}\in \widehat {\cal M}$ such that $\alpha A_{\hat{\mu}}$ is stable, in the sense that all its eigenvalues lie strictly within the unit circle.
\item[(c)] \emph{State Space Structure}: The state space $X$ has the property that for all $v\in \Re_+^n$, there exists an $x\in X$ and a scalar $s\geq0$ such that $sx= v$ (this is true in particular if $X=\re^n_+$).  
\item[(d)] \emph{Observability}: There exists an integer $N$ such that the optimal cost of the $N$-stage version of the problem [the problem of minimizing the cost $\sum_{k=0}^{N-1}\alpha^kg(x_k,u_k)$, starting from any nonzero initial state $x_0\in X$] is strictly positive.
 \end{itemize}
\end{assumption}

Part (a) above is the principal assumption and defines the semilinear character of the problem. Parts (b)-(d) are technical assumptions, whose significance will become clear from the analysis that follows. In particular, part (d) is called an \emph{observability} assumption because, in analogy with the standard notion of observability in control theory, it requires that every nonzero state leaves a detectable ``signature” in the cost: starting from any nonzero state, a strictly positive cost must be incurred within a finite horizon.

Note that from part (a) of Assumption~\ref{asm:dete_well} and Eqs.~\eqref{eq:semil_cond}, \eqref{eq:common_policy_min}, we have
   \begin{equation}
   \label{eq:dete_parta_lin}
       \hat c'x=G_\mu(c)'x=\min_{u\in U(x)}\big[g(x,u)+\alpha c'f(x,u)\big],\qquad \text{for all }x\in X,
   \end{equation}
    where $\m\in\widehat {\cal M}$ is the policy that attains the minimum in Eq.~\eqref{eq:common_policy_min}, and $G_\mu:\Re_+^n\mapsto\Re_+^n$ denotes the mapping 
    \begin{equation}
        \label{eq:g_mu_def}
        G_\mu(c)=q_\mu +\alpha A_\mu'c.
    \end{equation}
Note also that the VI algorithm applied to functions $J\in \widehat{\cal J}$ of the form $J(x)=c'x$ defines uniquely a function $G:\Re^n_+\mapsto\Re^n_+$ through the equation 
    \begin{equation}
        \label{eq:bellman_op_dete}
        G(c)'x=\min_{u\in U(x)}\big[g(x,u)+\alpha c'f(x,u)\big]=\min_{\mu\in\widehat {\cal M}}{G_{\mu}(c)'x},\qquad \text{for all }x\in X,
    \end{equation}
cf. Eq.~\eqref{eq:dete_parta_lin}. We can view this as the Bellman equation of the problem, restricted to the class of functions $\widehat{\cal J}$. Thus, the present framework is tailored to problems for which the optimal cost function belongs to the class $\widehat{\cal J}$ of nonnegative linear functions, and in the subsequent analysis we show that this property indeed holds under Assumption~\ref{asm:dete_well}.

There are many practical problems that fall into the framework considered here. In what follows, we provide a few examples where part (a) of Assumption~\ref{asm:dete_well} is satisfied. The set of policies $\widehat{\cal M}$ that forms a structured pair with $\widehat{\cal J}$, as in Assumption~\ref{asm:dete_well}(a), will be specified for each example.

\begin{example}[Control of Positive Bilinear Systems]\label{eg:bilinear}
Consider the case where the state $x$ consists of $n$ scalar components $x^1,\ldots,x^n$, and the control $u$ consists of $m$ scalar components $u^1,\ldots,u^m$. We assume that each component $u^i$ is constrained within a subset $U^i$.\footnote{In a more general formulation, each $u^i$ can also be a vector.} The state equation is
\begin{equation}
    \label{eq:bilinear}
    x_{k+1}=Ax_k+\big[f_1(u_k^1)\;f_2(u_k^2)\;\dots\; f_m(u_k^m)\big] Bx_k,
\end{equation}
where $A$ is an $n\times n$ nonnegative matrix, $f_i:U^i\mapsto \re^n_+$, $i=1,\dots,m$, are given functions, and $B$ is an $m\times n$ nonnegative matrix. The cost at stage $k$ is
$$q'x_k+\big[g_1(u^1)\;g_2(u^2)\;\dots\;g_m(u^m)\big]Bx_k,$$
where $q\in \re^n_+$ and $g_i:U^i\mapsto \Re_+$ are given vector and function, respectively. Note that  we allow that for some $i$, $f_i$ and $g_i$ are identically 0, thus eliminating the corresponding control components $u^i$. In this way we can model the case where the number of control components is smaller than $m$.

{A special case of Eq.~\eqref{eq:bilinear} is when $f_i(u^i)=u^iv_i$, with $v_i$ being some nonnegative vectors. In the literature, this form of the state equation is called a \emph{bilinear system}, where ``bilinear'' refers to the product terms $u^ix^j$ in Eq.~\eqref{eq:bilinear}.} Bilinear systems with nonnegative state variables have been used to address a variety of problems in medicine \cite{ledzewicz2002optimal}, biochemistry \cite{banks1975nonlinear}, and macroeconomics \cite{d1975bilinearity}. Moreover, a special form of positive bilinear system, known as a \emph{compartmental system}, has proved to be effective in the analysis and the control of infection
dynamics \cite{brauer2012mathematical,martcheva2015introduction,sharomi2017optimal}. The theory of optimal control of compartmental systems was recently developed in \cite{blanchini2023optimal} for a related continuous-time model with bilinear dependence on the control components and bounded interval constraints. The present example may be viewed as a discrete-time analogue that also allows more general nonlinear functions $f_i$ and $g_i$. For monographs focused on bilinear systems, see, e.g., \cite{mohler1973bilinear} and \cite{elliott2009bilinear}.  

We define $\widehat{\cal M}$ as the special set of policies $\mu$ such that the $\mu(x)$ is the same for all $x\in X$. Suppose that $\mu\in \widehat{\cal M}$ is applied and that the $i$th component of $\mu(x)$ is $u^i$, then it can be seen that the state equation and the cost at stage $k$ are $x_{k+1}=A_\mu x_k$ and $q_\mu'x_k$ respectively, where
$$A_\mu=A+\big[f_1(u^1)\;f_2(u^2)\;\dots\;f_m(u^m)\big]B,\; q_\mu'=q'+\big[g_1(u^1)\;g_2(u^2)\;\dots\;g_m(u^m)\big]B,$$
so that conditions given in Eq.~\eqref{eq:semil_cond} are satisfied. 

Let us now apply the VI algorithm starting with $J(x)=c'x$, $c\ge0$. It produces the function
$$\hat J(x)=(q+A'c)'x+\min_{\substack{u^i\in U^i\\i=1,\dots,m}}\Big[\big(g_1(u^1)+c' f_1(u^1)\big)\;\dots\;\big(g_m(u^m)+c' f_m(u^m)\big)\Big]Bx.$$
Since $Bx$ is nonnegative, it follows that the corresponding minimizing control is the same for all $x$, or equivalently, the minimizing policy $\mu$ belongs to $\widehat{\cal M}$. Moreover, we have $\hat J(x)=\hat c'x$ with
$$\hat c=q+A'c+D(c),$$
where $D(c)$ is the vector
$$D(c)=B'\big[d_1(c)\;d_2(c)\;\dots\;d_m(c)\big]',$$
with 
$$d_i(c)=\min_{u^i\in U^i}\big[g_i(u^i)+c'f_i(u^i)\big], \qquad i=1,\ldots,m.$$
Thus $\hat J\in \widehat{\cal J}$ and it can be seen that all the conditions of Assumption~\ref{asm:dete_well}(a) are satisfied.

\end{example}

\begin{example}[Control of Column Switching Systems]\label{eg:switch}
{Let us consider the special case of Example~\ref{eg:bilinear} with $m=n$ and $B$ equal to the identity matrix. The state equation \eqref{eq:bilinear} then simplifies to
\begin{equation}
    \label{eq:switch}
    x_{k+1}=Ax_k+\sum_{i=1}^nf_i(u_k^i)x_k^i,
\end{equation}
with stage cost
$$q'x_k+\sum_{i=1}^ng_i(u_k^i)x_k^i.$$

Problems of this form arise in several contexts, including those where the state represents a probability distribution, which we will discuss in Example~\ref{eg:mdp}. Equation~\eqref{eq:switch} also suggests a connection with the optimal control of switched systems, since the minimization over controls resembles selecting a switching action at each stage. However, our formulation applies only when the column vectors $f_i(u_k^i)$, multiplied by the individual components $x_k^i$, can be switched independently, and does not extend to general switched systems. In the latter case Assumption~\ref{asm:dete_well}(a) may fail, so Bellman’s equation may admit no solution within the class of linear functions. Thus, while there is a conceptual link, our framework does not encompass the broader class of positive switched linear systems studied in the literature (see, e.g., \cite{blanchini2008set,fornasini2011stability,blanchini2015switched,rantzer2015scalable}).}
\end{example}

\begin{example}[Positive Linear Systems with Control Constraints]\label{eg:pos_lin}
    Consider a problem where the state equation is 
    $$x_{k+1}=Ax_k+Bu_k,$$
    with $A$ and $B$ being $n\times n$ and $n\times m$ matrices, respectively. The cost of stage $k$ is
    $$q'x_k+r'u_k,$$
    with $q$ and $r$ being vectors in  $\re^n_+$ and $\re^m$, respectively. For every state $x$, the control $u$ is selected from the set 
    $$U(x)=\{u\in\re^m\,|\,|u|\leq Hx\},$$
    where $|u|$ is the vector whose components are the absolute values of the components of $u$, and $H$ is a given $m\times n$ matrix. Moreover, suitable assumptions are made regarding $A$, $B$, and $H$ so that $x_{k+1}$ remains in $\re^n_+$ regardless of the value of $x_k\in\re^n_+$ and the choice of $u_k\in U(x_k)$. 
    
    This problem was first studied by Rantzer \cite{rantzer2022explicit}, who showed that the optimal cost function can be obtained by using linear programming. The DP methodology for this problem, and another closely related problem, was developed by Li and Rantzer \cite{li2024exact}. Subsequently, this problem and some of its variants have been studied by Rantzer, Ohlin, Tegling, Gurpegui, Pates, Jeeninga, and Bencherki \cite{ohlin2023optimal,gurpegui2023minimax,pates2024optimal,bencherki2024data,gurpegui2024minimax,ohlin2024heuristic}.

   Let us define the set $\widehat{\cal M}$ as the set of linear policies:
    $$\widehat{\cal M}=\big\{\mu\,|\,\mu(x)=Lx,\hbox{ where $L$ is an $n\times m$ matrix and }|L|\leq H\big\},$$
    where $|L|$ is the $n\times m$ matrix whose components are the absolute values of the components of $L$. It can be seen that if $\mu\in \widehat{\cal M}$ so that $\mu(x)=Lx$ for some $L$, then $\mu(x)\in U(x)$ for all $x$. Moreover, the state equation and the cost at stage $k$ under the policy $\mu$ are $x_{k+1}=A_\mu x_k$ and $q_\mu'x_k$ respectively, where
    $$A_\mu=A+BL,\quad q_\mu=q+L'r.$$

  Starting with $J(x)=c'x$, $c\ge0$, the VI algorithm produces the function
    \begin{equation}
    \label{eq:bellman_pos_lin}
        \hat{J}(x)=(q+A'c)x+\min_{|u|\leq Hx}(r+B'c)'u.
    \end{equation}
    Let us denote by $b_i$ the $i$th column of $B$ and by $h_i'$ the $i$th row of $H$. It can be seen that the minimum in Eq.~\eqref{eq:bellman_pos_lin} is attained at $Lx$, where
    $$L=-\begin{bmatrix}
\sign(r_1+b_1'c)h_1'\\
\vdots\\
\sign(r_m+b_m'c)h_m',
\end{bmatrix}$$
and $\sign(\cdot)$ is the function that takes the value $1$ if its argument is nonnegative and $-1$ otherwise.
    As a result, we have $\hat{J}(x)=\hat{c}'x$, where
    $$\hat{c}=q+L'r+(A+BL)'c.$$
    Therefore, starting with a linear function $J\in \widehat{\cal J}$, VI produces another linear function $\hat{J}\in \widehat{\cal J}$, and it can be seen that all the conditions of Assumption~\ref{asm:dete_well}(a) are satisfied.
\end{example}

\begin{example}[Markov Decision Problems with Distributions as States]\label{eg:mdp}
Consider the case where each state is a probability distribution over a finite set that consists of $n$ points. Thus, each state $x$ is a column vector consisting of $n$ scalar components $x^1,\dots,x^n$, where $x^i$ is the probability of point $i$. Each control $u$ also has $n$ scalar components $u^1,\dots,u^n$, where each $u^i$ is chosen from a subset $U^i$. Given the current state $x_k$, the state at time $k+1$ is given by  
\begin{equation}
    \label{eq:mdp}
    x_{k+1}=\sum_{i=1}^np_i(u^i_k)x^i_k,
\end{equation}
where the function $p_i$ maps each $u^i$ to a probability distribution. Given an initial distribution $x_0$, the objective is to minimize the total cost
$$\sum_{k=0}^\infty \alpha^k \sum_{i=1}^n g_i(u_k^i)x_k^i,$$
where $\alpha\in (0,1)$, $g_i:U^i\mapsto \re_+$, $i=1,\dots,n$, and $x_k$ evolves according to the state equation \eqref{eq:mdp}. 

Let us consider the set $\widehat{\cal M}$ that consists of all policies $\mu$ such that $\mu(x)$ is the same for all $x\in X$. When applying a policy $\mu\in \widehat{\cal M}$ such that the $i$th component of $\mu(x)$ is $u^i$, the state equation and the cost at stage $k$ are $x_{k+1}=A_\mu x_k$ and $q_\mu'x_k$ respectively, where
$$A_\mu=\big[p_1(u^1)\;p_2(u^2)\;\dots\;p_n(u^n)\big],\quad q_\mu'=\big[g_1(u^1)\;g_2(u^2)\;\dots\;g_n(u^n)\big],$$
so that the conditions of Eq.~\eqref{eq:semil_cond} are satisfied. 

Now we apply the VI algorithm starting with $J(x)=c'x$, where $c$ is nonnegative. We obtain a new function
$$\hat{J}(x)= \sum_{i=1}^n\min_{u^i\in U^i}[g_i(u^i)+\alpha c' p_i(u^i)]x^i.$$
Using a derivation similar to that of Example~\ref{eg:bilinear}, we can show that $\hat{J}(x)=\hat{c}'x$ for some $\hat{c}\in \re^n_+$.

Systems whose states are probability distributions arise in partially observed Markov decision problems (POMDP), and other interesting contexts in DP. For example, Gao et al. \cite{gao2023distributional} studied the evolution of distributions over time using a DP formulation. In another theoretically interesting context, the measurability issues in stochastic optimal control were addressed by Bertsekas and Shreve (\cite[Chapter~9]{bertsekas1978stochastic}) using a Markovian decision framework, where states were modeled by probability distributions.The present example is a direct application of the approach in \cite[Chapter~9]{bertsekas1978stochastic}. In the same spirit, the semilinear DP framework extends naturally to stochastic shortest path problems with nonnegative costs, formulated with probability distributions as states, and to their special cases such as \cite{todorov2006linearly}. 
\end{example}

To set the stage for our analysis, we will now state some results that hold for the general nonnegative cost deterministic problem \eqref{eq:proq_dete},  even without Assumption~\ref{asm:dete_well}. Formal statements of these results, for the broader context of stochastic problems, are provided in Appendix~\ref{app:a}.

We denote by $J^*(x)$ the optimal cost starting from a state $x\in X$. We know that $J^*$ is a solution of Bellman's equation, which takes the form
\begin{equation}
    \label{eq:bellman_dete_j}
    J(x)=\min_{u\in U(x)}\Big[g(x,u)+\alpha J\big(f(x,u)\big)\Big],\qquad x\in X.
\end{equation}
For a given policy $\mu$, we denote by $J_{\mu}(x_0)$ the cost starting from a state $x_0\in X$ and using $\mu$, i.e., 
\begin{equation}
    \label{eq:j_mu_def}
    J_\mu(x_0)=\sum_{k=0}^{\infty} \alpha^k g\big(x_k,\mu(x_k)\big)\quad \text{for all }x_0,
\end{equation}
where $x_{k+1}=f\big(x_k,\mu(x_k)\big)$, $k=0,1,\dots$. Similar to $J^*$, $J_{\mu}$ is a solution of the corresponding Bellman's equation for policy $\mu$, 
\begin{equation}
    \label{eq:bellman_dete_mu}
    J(x)=g\big(x,\mu(x)\big)+\alpha J\Big(f\big(x,\mu(x)\big)\Big),\qquad x\in X.
\end{equation} 
We say that a policy $\mu^*$ is \emph{optimal} if $J_{\mu^*}(x)=J^*(x)$ for all $x$. It is well known that $\mu^*$ is optimal if and only if $\mu^*(x)$ attains the minimum in Eq. \eqref{eq:bellman_dete_j} for all $x\in X$, with $J^*$ in place of $J$. Note that for the policy $\hat\mu$ that satisfies Assumption~\ref{asm:dete_well}(b) ($\alpha A_{\hat \mu}$ is stable), we have
\begin{equation}
    \label{eq:boundedness}
J^*(x)\le J_{\hat\mu}(x)<\infty,\qquad \hbox{for all }x\in X.
\end{equation}
Our analysis of the next two sections will revolve around the uniqueness of solution of  Bellman's equation, the existence of optimal policies within the class $\widehat{\cal M}$, and the convergence properties of the VI and PI algorithms. 

\section{Bellman's Equation and Optimal Policies}\label{sec:bellman_dete}
Our analysis will be based primarily on the VI algorithm, which takes the form
\begin{equation}
    \label{eq:vi_dete}
    J_{k+1}(x)=\min_{u\in U(x)}\Big[g(x,u)+\alpha J_{k}\big(f(x,u) \big)\Big],\qquad k=0,1,\ldots,
\end{equation}
starting from some initial nonnegative function $J_0$. Its convergence properties are summarized in Appendix~\ref{app:a}.

When the VI algorithm is specialized to our problem starting with $J_0(x)=c_0'x$ with $c_0\ge0$, it takes the form 
$$J_{k+1}(x)=q_\mu' x+\alpha c_k'A_\mu x=(q_\mu+\alpha A_\mu'c_k)'x,\qquad \text{for all }x\in X,$$
for some $\mu\in \widehat{\cal M}$ [cf.\ Assumption~\ref{asm:dete_well}(a)]. Equivalently, we have 
$$J_{k+1}(x)=G(c_k)'x,$$
where $G$ is uniquely defined via
$$G(c)'x=\min_{\mu\in \widehat{\cal M}}G_\mu(c)'x,\quad \text{for all }x\in X,$$
cf. Eq. \eqref{eq:bellman_op_dete}. As a result, we have that $J_{k+1}(x)=c_{k+1}'x$ with $c_{k+1}=G(c_k)$. 

Our analysis relies on a key \emph{monotonicity} property of the functions  $G$ and $G_\mu$, $\mu\in \widehat{\cal M}$, which states that if two vectors $c,\Bar{c}\in \re_+^n$ satisfy $c\leq \Bar{c}$, then
$$G(c)\leq G(\Bar{c}), \qquad G_\mu(c)\leq G_\mu(\Bar{c}),\ \ \text{for all }\mu\in \widehat{\cal M}.$$
The monotonicity of $G_\mu$ follows from its definition $G_\mu(c)=q_\mu+\alpha A_\mu'c$ [cf. Eq.~\eqref{eq:g_mu_def}] and the fact that $A_\mu$ is nonnegative. To see the monotonicity of $G$, we first note that 
\begin{equation}
    \label{eq:monotone_g_element}
    G(c)'x=\min_{\mu\in \widehat{\cal M}}G_\mu(c)'x\leq \min_{\mu\in \widehat{\cal M}}G_\mu(\Bar{c})'x=G(\Bar{c})'x,\quad \text{for all }x\in X,
\end{equation}
where the inequality follows from the monotonicity of $G_\mu$ for all $\mu\in \widehat{\cal M}$. Assumption~\ref{asm:dete_well}(c) implies that for every $i=1,2,\dots,n$, there exists some $x\in X$ with its $i$th component being positive and all the other components being zero. As a result, using  condition \eqref{eq:monotone_g_element}, we have $G(c)\leq G(\Bar{c})$, showing the monotonicity property of $G$.

Another property of $G$ that we will use is that for all $c\in \re_+^n$, 
\begin{equation}
    \label{eq:g_le_gm}
    G(c)\leq G_\mu(c),\qquad \hbox{for all }\mu\in \widehat{\cal M}.
\end{equation}
This inequality follows from Eq.~\eqref{eq:monotone_g_element}, by  considering states $x\in X$ with only one component being nonzero.

We will now use the VI algorithm to establish the linearity and the uniqueness of solution of Bellman's equation. 

\begin{proposition}\label{prop:unique_dete}
There exists a positive vector $c^*\in \Re^n_+$ such that $J^*(x)=(c^*)'x$ for all $x\in X$. Moreover, $c^*$ is the unique vector within $\Re_+^n$ that satisfies
    \begin{equation}
        \label{eq:bellman_dete}
        c^*=G(c^*).
    \end{equation}    
\end{proposition}

\begin{proof}
Our proof will proceed in two steps. First, we will consider the sequence of functions $\{J_k\}$ generated by VI, starting with the function $J_0(x)\equiv0$. Based on parts (a)-(c) of Assumption~\ref{asm:dete_well}, particularly the monotonicity and continuity of the function $G$, we will show that the sequence $\{J_k\}$ is composed of linear functions $J_k(x)=c_k'x$, so that $J_k\in \widehat{\cal J}$ and $c_{k+1}=G(c_k)$. Moreover, we will show that $\{c_k\}$ converges and its limit, denoted by $c^*$, defines the optimal cost function, i.e., $J^*(x)=(c^*)'x$ for all $x$. Finally, using Assumption~\ref{asm:dete_well}(d), we will show that $c^*$ is the unique solution of the equation $c=G(c)$ within $\re_+^n$, and that $c^*>0$.  

Consider the sequence of functions $\{J_k\}$ generated by VI [cf. Eq.~\eqref{eq:vi_dete}] with $J_0\equiv0$. As noted above, we have that $J_k(x)=c_k'x$ for $c_k\in \Re_+^n$, $c_{k+1}=G(c_k)$, and $c_0=0$; cf. Eq.~\eqref{eq:bellman_op_dete}. Since $J_0(x)\leq J^*(x)$ for all $x$, applying VI on both sides, and using the fact that $J^*$ is a solution of Bellman's equation, we have $J_1(x)\leq J^*(x)$ for all $x$, and similarly by induction, $c_k'x=J_k(x)\leq J^*(x)$ for all $k$ and $x$. Since $J^*(x)\leq J_{\hat \mu}(x)<\infty$, where $\hat\mu$ is the policy of Assumption~\ref{asm:dete_well}(b), the sequence $\{c_k\}$ is upper-bounded. Moreover, it is monotonically increasing since $G$ is monotone. Therefore, $\{c_k\}$ converges to a vector in $\re^n_+$, which we denote by $c^*$. Taking limit on both sides of the equation $c_{k+1}=G(c_k)$, and using our assumption that $G(c)$ depends continuously on $c$, we have $c^*=G(c^*)$, or equivalently,
$$J_\infty(x)=\inf_{u\in U(x)}\big\{g(x,u)+\alpha J_\infty\big(f(x,u)\big)\big\},$$
where $J_\infty(x)=(c^*)'x$. Thus the nonnegative function $J_\infty$ is a solution to Bellman's equation, and we have $J_\infty\leq J^*$. On the other hand, it is known that every nonnegative function that solves Bellman's equation is lower-bounded by $J^*$; see Prop.~\ref{prop:classical}(a) in Appendix~\ref{app:a}. Therefore, we have $J_\infty\geq J^*$, and it follows that $J_\infty=J^*$, or equivalently, $J^*(x)=(c^*)'x$ for all $x$.

To show the uniqueness part, let $\Bar{c}\ge0$ be a solution to the equation $c=G(c)$. Then the function $\Bar{c}'x$ also solves Bellman's equation. Since $J^*(x)=(c^*)'x$ is a lower bound to all solutions of Bellman's equation, using Assumption~\ref{asm:dete_well}(c), we have that $c^*\leq \Bar{c}$. 

Next, we claim that $c^*>0$. Indeed, $N$ iterations 
 of the VI algorithm starting from 0 produce the cost function of the $N$-stage problem, which is a lower bound to $J^*$, and  by  Assumption~\ref{asm:dete_well}(d), is positive for all $x\ne0$. Hence we have $J^*(x)=(c^*)'x>0$ for all $x\ne0$.  By considering states $x$ with a single component being nonzero, it follows that $c^*>0$. 
 
The inequality $c^*>0$ implies that for some $s>1$ we have $c^*\leq \Bar{c}\leq sc^*$, so that
\begin{equation}
        \label{eq:bellman_yuineq}J^*(x)={(c^*)}'x\le\Bar{c}'x\le s{(c^*)}'x=sJ^*(x).
\end{equation}
Consider the VI algorithm starting from $J_0=\Bar{c}'x$. It produces a sequence $\{J_k\}$ that is identically equal to $\Bar{c}'x$ (since $\Bar{c}'x$ is a solution to Bellman's equation), and  converges to $J^*$ [by Eq.\ \eqref{eq:bellman_yuineq} and Prop.~\ref{prop:classical}(e)]. It follows that $\Bar{c}'x={(c^*)}'x$ for all $x\in X$.
By considering states $x$ with a single component being nonzero, we obtain $\Bar c=c^*$.
\end{proof}

In the next proposition, we will show the existence of at least one optimal policy $\mu^*$ within $\widehat{\cal M}$ such that $\alpha A_{\mu^*}$ is stable.

\begin{proposition}
    \label{prop:optimal_policy}
    There exists an optimal policy $\mu^*$ that belongs to $\widehat{\cal M}$ and is such that $\alpha A_{\mu^*}$ is stable. 
\end{proposition}
\begin{proof}
    In view of Prop.~\ref{prop:unique_dete}, we have that $J^*(x)=(c^*)'x$ where $c^*$ is the unique solution to $c=G(c)$ within $\Re^n_+$. Consider a policy $\mu^*\in \widehat{\cal M}$ that satisfies 
    $$\mu^*(x)\in \arg\min_{u\in U(x)}\big\{g(x,u)+\alpha (c^*)'f(x,u)\big\},\qquad \text{for all }x\in X.$$
    Such a policy exists by Assumption~\ref{asm:dete_well}(a), and  is optimal by Prop.~\ref{prop:classical}(c) in Appendix~\ref{app:a}. As a result, for all $x\in X$ with $x\ne 0$, we have that 
    \begin{equation}
        \label{eq:stable_cost}
        0<J^*(x)=J_{\mu^*}(x)=\lim_{\ell\to\infty}(q_{\mu^*})'\sum_{i=0}^{\ell}(\alpha A_{\mu^*})^ix=J^*(x)<\infty.
    \end{equation}
    
    To show that the matrix $\alpha A_{\mu^*}$ is stable, we note that by the Perron-Frobenius theorem (Prop.~\ref{prop:f_p} in Appendix~\ref{app:a}), $\alpha A_{\mu^*}$ has a maximal real nonnegative eigenvalue $\lambda$, with corresponding eigenvector $v\ge0$, $v\ne0$. Assume, to arrive at a contradiction, that $\lambda\geq1$. Let $\bar v$ be a positive multiple of $v$ which belongs to $X$, cf. Assumption~\ref{asm:dete_well}(c). 
   Then from Eq.~\eqref{eq:stable_cost},
    \begin{equation}
        \label{eq:eigen_cost}
        0<\lim_{\ell\to\infty}(q_{\mu^*})'\sum_{i=0}^{\ell}(\alpha A_{\mu^*})^i\bar v=\bigg(\sum_{i=0}^{\infty}\lambda^i\bigg)(q_{\mu^*})'\bar v<\infty.
    \end{equation}
Since $\lambda\geq1$, it follows that  $(q_{\mu^*})'\bar v=0$. 
 On the other hand, using also Eq.~\eqref{eq:g_le_gm},  we have
    \begin{equation}
    \label{eq:stable_positive}
        \big(G^{N}(0)\big)'\bar v\le (q_{\mu^*})'\sum_{i=0}^{N-1}(\alpha A_{\mu^*})^i\bar v=\bigg(\sum_{i=0}^{N-1}\lambda^i\bigg)(q_{\mu^*})'\bar v=0,
    \end{equation}
    where $G^{N}$ denote the $N$-fold composition of $G$. Since $\big(G^{N}(0)\big)'\bar v$ is the $N$-stage cost starting from the nonzero initial state $\bar v$, this contradicts Assumption~\ref{asm:dete_well}(d). Therefore, we have $\lambda<1$, i.e., the matrix $\alpha A_{\mu^*}$ is stable.  
\end{proof}

\section{Value and Policy Iteration Algorithms}\label{sec:algo_dete}

In this section, we show how various implementations of the VI and PI algorithms can be used to compute  the optimal cost vector $c^*$ and an optimal policy.
In addition, we will provide an alternative approach to compute $c^*$ through the solution of a convex program. 

\subsection{Synchronous and Asynchronous Value Iteration}
In general, the VI algorithm of Eq.\ \eqref{eq:vi_dete} generates a sequence of functions $J_k$. However, for our semilinear problem, the functions $J_k$ generated by VI can be fully specified by their parameter vectors $c_k$. As a result, the VI algorithm can be described in terms of $c_k$ as follows: 
\begin{equation}
    \label{eq:vi_dete_syn}
    c_{k+1}=G(c_k),
\end{equation}
where $c_0\in \Re_+^n$ is the initial condition.
The next proposition shows that the sequence $\{c_k\}$ converges to $c^*$.

\begin{proposition}
    \label{prop:vi_dete}
    The sequence $\{c_k\}$ generated by the VI algorithm \eqref{eq:vi_dete_syn} converges to $c^*$, starting with any initial vector $c_0\in \Re^n_+$.
\end{proposition}

\begin{proof}
    Let $\{c_k\}$ be a sequence generated by the VI algorithm \eqref{eq:vi_dete_syn}, starting with with some $c_0\in \Re^n_+$. Since by Prop.~\ref{prop:unique_dete}, $c^*>0$, we can find $s>1$ such that $c_0\leq sc^*$. Consider the sequences $\{\underline{c}_k\}$ and $\{\overline{c}_k\}$ generated by VI with $\underline{c}_0=0$ and $\overline{c}_0=sc^*$, respectively. Then $\underline{c}_0\leq c_0\leq \overline{c}_0$, and by the monotonicity of $G$, we have 
    $$\underline{c}_k\leq c_k\leq \overline{c}_k,\quad k=0,1,\dots.$$
    Using an argument similar to the one of the proof of Prop.~\ref{prop:unique_dete}, we have that $\underline{c}_k\to c^*$ and $\overline{c}_k\to c^*$. It follows that $c_k\to c^*$.
\end{proof}

From the proof of Prop.~\ref{prop:unique_dete}, it can be seen that even without Assumption~\ref{asm:dete_well}(d), we can still show that the sequence $\{c_k\}$ generated by VI converges to $c^*$, under the additional assumption that $0\leq c_0\leq c^*$. Indeed, in the proof of Prop.~\ref{prop:unique_dete}, we have shown that $\underline{c}_k\to c^*$, where $\underline{c}_0=0$, and $\underline{c}_{k+1}=G(\underline{c}_k)$. Therefore, by the monotonicity of $G$, we also have $\underline{c}_k\leq c_{k}\leq G^k(c^*)=c^*$. This yields that $c_k\to c^*$. 

The VI algorithm updates all the components of $c_k$ simultaneously at every iteration. In the literature, this is often referred to as a \emph{synchronous algorithm}. An alternative class of algorithms, called \emph{asynchronous}, updates only \emph{some} components of $c_k$ at each iteration. The asynchronous VI algorithm was first developed in \cite{bertsekas1982distributed}, and it was extended to  solve more general fixed point problems in \cite{bertsekas1983distributed}; see \cite{bertsekas2012dynamic} and \cite{bertsekas2022abstract} for recent textbook discussions. In what follows, we will develop the asynchronous version of the VI algorithm for our semilinear problem.

Given a vector $c\in \Re_+^n$, let us denote by $c(i)$ its $i$th element. We consider a partition of the set $I=\{1,2,\dots,n\}$ into the sets $I_1,\dots,I_m$, and a corresponding partition $c=(c^1,\dots,c^m)$, where $c^\ell$ is the restriction of $c$ on the set $I_\ell$. We associate with each processor $\ell$ a set of iteration indices $\mathcal{R}_\ell \subset \{0,1,\dots\}$. In the asynchronous VI algorithm, processor $\ell$ updates $c^\ell$ only at iterations $k \in \mathcal{R}_\ell$, using components $c^j$, $j\neq \ell$, received from other processors. Specifically, the value of $c^j$ available to processor $\ell$ at iteration $k$ is the one computed by processor $j$ at iteration $\tau_{\ell j}(k)\in \{0,1,\dots\}$. Here the first subscript $\ell$ in $\tau_{\ell j}(k)$ denotes the receiving processor, the second subscript $j$ denotes the sending processor, and $k-\tau_{\ell j}(k)$ represents the communication ``delay.'' With this notation, the asynchronous VI algorithm is defined as
\begin{equation}
    \label{eq:asy_vi_dete}
    c^\ell_{{k}+1}(i)=\begin{cases}G\Big(c^1_{\tau_{\ell 1}({k})},\dots ,c^m_{\tau_{\ell m}({k})}\Big)(i) &\text{if }{k}\in \mathcal{R}_\ell,\;i\in I_\ell,\\
    c^\ell_{{k}}(i)&\text{if }{k}\not\in \mathcal{R}_\ell,\;i\in I_\ell.
    \end{cases}
\end{equation} 
To ensure that the information received by each processor is sufficiently ``new'' in order to ensure algorithmic convergence, we make the following assumption, which is known as the \emph{total asynchronism} condition; see \cite[p.~430]{bertsekas1989parallel}.

\begin{assumption}[Continuous Updating and Information Renewal]\label{asm:asy_dete}
\begin{itemize}
    \item[(a)] The set of iteration indices $\mathcal{R}_\ell$ at which processor $\ell$ updates $c^\ell$ is infinite for each $\ell=1,\dots,m$.
    \item[(b)] $\lim_{k\to\infty}\tau_{\ell j}(k)=\infty$ for all $\ell,j=1,\dots,m$. 
\end{itemize}
\end{assumption}

In Prop.~\ref{prop:unique_dete}, we have shown that $c^*$ is the unique fixed point of the function $G$ within $\re_+^n$. With the additional Assumption~\ref{asm:asy_dete}, we have the following convergence result.

\begin{proposition}\label{prop:asy_vi_dete}
    Let Assumption~\ref{asm:asy_dete} hold. Then the sequence $\{c_k\}$ generated by the asynchronous VI algorithm \eqref{eq:asy_vi_dete} converges to $c^*$, starting with any initial vector $c_0\in \Re^n_+$. 
\end{proposition}

Prop.~4 follows by a direct application of the \emph{asynchronous convergence theorem}, first established in \cite{bertsekas1983distributed}; also see \cite[Section~6.2]{bertsekas1989parallel} and \cite[Section~2.6]{bertsekas2012dynamic}. For this reason, we only provide a brief discussion of the proof ideas. To apply the theorem, it is necessary to construct a sequence of nonempty sets $S(k)\subseteq \Re_+^n$ satisfying the following properties:  
1) $S(k+1)\subseteq S(k)$ for all $k=0,1,\dots$;  
2) If a sequence $\{\hat c_k\}$ satisfies $\hat c_k \in S(k)$ for all $k$, then $\hat c_k $ converges to $c^*$;  
3) For all $k$ and $c\in S(k)$, we have $G(c)\in S(k+1)$;  
4) Each set $S(k)$ has a Cartesian product structure $S(k)=S_1(k)\times\cdots\times S_m(k)$, where $S_\ell(k)\subseteq \Re_+^{n_\ell}$ and $n_\ell$ is the dimension of $c^\ell$;  
5) The initial set $S(0)$ contains $c_0$. Under these conditions and Assumption~4.1, the asynchronous convergence theorem ensures that $c_k$ converges to $c^*$.  

To this end, define the sequence of sets $S(k)=\{c\,|\,\underline{c}_k\leq c\leq \overline{c}_k\}$, where the bounding vectors $\underline{c}_k$ and $\overline{c}_k$ are given iteratively by $\underline{c}_0=0$, $\overline{c}_0=sc^*$ for some constant $s>1$, and for all $k$, $\underline{c}_{k+1}=G(\underline{c}_k)$, and $\overline{c}_{k+1}=G(\overline{c}_k)$. Given any initial vector $c_0\in \Re_+^n$, one can select $s>1$ such that $c_0\leq sc^*$, in view of $c^*>0$ [cf. Prop.~1]. It is straightforward to verify that the sets $\{S(k)\}$ satisfy the conditions above, which completes the proof. We refer to \cite[Sections~2.6.1, 3.6.1]{bertsekas2022abstract} for a related discussion.

\subsection{Classical and Optimistic Policy Iteration}

The PI algorithm starts with a policy $\mu^0$ and generates a sequence of policies $\{\mu^k\}$ by first performing the \emph{policy evaluation} step, which computes its cost function $J_{\m^k}$, defined pointwise by
\begin{equation}
\label{eq:j_muk_def}
    J_{\m^k}(x_0)=\sum_{\ell=0}^{\infty} \alpha^\ell g\big(x_\ell,\mu^k(x_\ell)\big),\qquad \text{for all }x_0,
\end{equation}
where $x_{\ell+1}=f\big(x_\ell,\mu^k(x_\ell)\big)$, $\ell=0,1,\dots$; cf. Eq.~\eqref{eq:j_mu_def}. This is followed by the \emph{policy improvement} step, which computes the policy $\mu^{k+1}$ through the minimization operation
\begin{equation}
    \label{eq:policy_imp}
    \m^{k+1}(x)\in\arg\min_{u\in U(x)}\Big[g(x,u)+\alpha J_{\m^k}\big(f(x,u)\big)\Big],\qquad \text{for all }x.
\end{equation}

For the semilinear problem considered here, the PI algorithm can be carried out in terms of the parameter vectors associated with the cost functions. We start with a policy $\mu^0\in \widehat{\cal M}$ such that $\alpha A_{\mu^0}$ is stable. At a typical iteration $k$, we have computed a policy $\mu^k\in\widehat{\cal M}$ with  $\alpha A_{\mu^k}$ stable and cost function given by
$$J_{\m^k}(x)=\sum_{\ell=0}^{\infty}  q_{\mu^k}'(\alpha A_{\mu^k})^\ell x,\qquad \text{for all }x;$$
cf. Eq.~\eqref{eq:j_muk_def}. Equivalently, we have $J_{\m^k}(x)=c_{\mu^k}'x$, where
\begin{equation}
    \label{eq:pi_eva_pos}
    c_{\mu^k}'=q_{\mu^k}'(I-\alpha A_{\mu^k})^{-1}.
\end{equation}
Moreover, the improved policy $\mu^{k+1}$ [cf.\ Eq. \eqref{eq:policy_imp}] belongs to $\widehat{\cal M}$ and satisfies
\begin{equation}
    \label{eq:pi_imp_pos}
    G_{\mu^{k+1}}(c_{\mu^k})=G(c_{\mu^k}).
\end{equation}

The following proposition deals with the convergence properties of the preceding PI algorithm. 
\begin{proposition}
    \label{prop:pi_dete}
    The PI algorithm \eqref{eq:pi_eva_pos}-\eqref{eq:pi_imp_pos} is well-defined, i.e., for every $k$, $\mu^k\in \widehat{\cal M}$ and $\alpha A_{\mu^k}$ is stable. Moreover, we have $c_{\mu^k}\to c^*$ as $k\to\infty$. If in addition $\widehat{\cal M}$ consists of a finite number of policies, then there exists some $\Bar{k}$ such that for all $k\geq \Bar{k}$, the policies $\mu^k$ are optimal.
\end{proposition}

\begin{proof}
    Our proof will proceed in three steps. First, we will show that the PI algorithm is well-posed in the sense that the inverse in Eq.~\eqref{eq:pi_eva_pos} exists, and that the policy improvement step of Eq.~\eqref{eq:pi_imp_pos} is possible. Next, we will show that the cost vector sequence $c_{\m^{k}}$ converges to $c^*$ by comparing it with a sequence generated by VI. Finally, we will show finite termination when $\widehat{\cal M}$ consists of finitely many policies.

    First, we note that the sequence of functions $\{J_{\mu^k}\}$ generated by PI is monotonically decreasing, i.e., $J_{\mu^{k+1}}\leq J_{\mu^k}$ for all $k$; see Prop.~\ref{prop:classical}(d). Since the initial policy $\mu^0$ is assumed to be such that $\alpha A_{\mu^0}$ is stable, the cost $J_{\mu^0}(x)$ is finite for all $x$, so we have that  $J_{\mu^k}(x)$ is finite for all $x$ and $k$. Using arguments similar to those in the proof for Prop.~\ref{prop:optimal_policy}, we can show that the matrices $\alpha A_{\mu^k}$ are stable for all $k$. As a result, the inverse in Eq.~\eqref{eq:pi_eva_pos} is defined for all $\mu^k$. Moreover, Assumption~\ref{asm:dete_well}(a) implies that there exists some $\mu^{k+1}\in \widehat{\cal M}$ that satisfies Eq.~\eqref{eq:pi_imp_pos}.

    To see that $\{c_{\mu^k}\}$ converges to $c^*$, we consider the auxiliary sequence $\{\Bar{c}_k\}$ with $\Bar{c}_0=c_{\mu^0}$, and $\Bar{c}_{k+1}=G(\Bar{c}_k)$. We will show by induction that
    $$c_{\mu^k}\leq \Bar{c}_k,\quad k=0,1,\dots.$$
    The inequality holds for $k=0$ in view of the definition of $\Bar{c}_0$. Suppose that $c_{\mu^k}\leq \Bar{c}_k$. Then by the monotonicity of $G$, we have 
    \begin{equation}
        \label{eq:pi_norm_ineq4}
        G(c_{\mu^k})\leq G(\Bar{c}_k)=\Bar{c}_{k+1}.
    \end{equation}
    In addition, we have 
    \begin{equation}
        \label{eq:pi_ineq}
        G_{\mu^{k+1}}(c_{\mu^k})=G(c_{\mu^k})\leq G_{\mu^k}(c_{\mu^k})=c_{\mu^k},
    \end{equation}
    where  the first equality holds by the definition of $\mu^{k+1}$, the inequality is due to the relations between $G$ and $G_{\mu^k}$ [cf. Eq.~\eqref{eq:g_le_gm}], and the second equality corresponds to Bellman's equation with respect to $\mu^k$; see Prop.~\ref{prop:classical}(b). From Eq. \eqref{eq:pi_ineq}, we have
    \begin{equation}
        G_{\mu^{k+1}}(c_{\mu^k})\leq c_{\mu^k}.
    \end{equation}
    Applying $G_{\mu^{k+1}}$ on both sides of this equation and using the monotonicity of $G_{\mu^{k+1}}$, we have 
    $$G_{\mu^{k+1}}\big(G_{\mu^{k+1}}(c_{\mu^k})\big)\leq G_{\mu^{k+1}}(c_{\mu^k}),$$
    or equivalently,
    $$q_{\mu^{k+1}}'x+\alpha \hat{J}\big(A_{\mu^{k+1}}x\big)\leq \hat{J}(x),\qquad \hbox{for all } x\in X,$$
    with $\hat{J}(x)=G_{\mu^{k+1}}(c_{\mu^k})'x$. Then in view of Prop.~\ref{prop:classical}(b), we have that
    \begin{equation}
        \label{eq:pi_norm_ineq2}
        c_{\mu^{k+1}}\leq G_{\mu^{k+1}}(c_{\mu^k}).
    \end{equation}
    Combining Eqs.\ \eqref{eq:pi_norm_ineq4} and \eqref{eq:pi_norm_ineq2}, and using the equality $G_{\mu^{k+1}}(c_{\mu^k})=G(c_{\mu^k})$, we obtain $c_{\mu^{k+1}}\leq \Bar{c}_{k+1}$. Since $\Bar{c}_{k}\to c^*$ by Prop.~\ref{prop:vi_dete}, and $c_{\mu^{k}}\geq c^*$ by the definition of $c^*$, we have that $c_{\mu^{k}}\to c^*$.

    Suppose that $\widehat{\cal M}$ is finite. Since $c_{\mu^{k+1}}\leq c_{\mu^k}$, then either $c_{\mu^{k+1}}\leq c_{\mu^k}$ and $c_{\mu^{k+1}}\neq c_{\mu^k}$, or $c_{\mu^{k+1}}= c_{\mu^k}$. The first case implies that $\mu^{k}\neq \mu^{k+1}$, which can only occur finitely often, since $\widehat{\cal M}$ is finite. Let $\Bar{k}$ be the smallest index such that $c_{\mu^{k+1}}= c_{\mu^k}$. Then we have
    $$c_{\mu^{\Bar{k}+1}}=G_{\mu^{\Bar{k}+1}}(c_{\mu^{\Bar{k}+1}})=G_{\mu^{\Bar{k}+1}}(c_{\mu^{\Bar{k}}})=G(c_{\mu^{\Bar{k}}})=G(c_{\mu^{\Bar{k}+1}}),$$
    where the first equality follows from Prop.~\ref{prop:classical}(b), and the second and last equalities follow from the definition of $\Bar{k}$. The third equality is due to the definition of $\mu^{\Bar{k}+1}$. Thus, we have  $c_{\mu^{\Bar{k}+1}}=G(c_{\mu^{\Bar{k}+1}})$. Since $c^*$ is the unique solution to $c=G(c)$, we obtain $c_{\mu^{\Bar{k}+1}}=c_{\mu^{\Bar{k}}}=c^*$. In other words, $\mu^{\Bar{k}}$ and $\mu^{\Bar{k}+1}$ are both optimal. Moreover, all the policies $\mu^k$, with $k\geq \Bar{k}+1$, satisfy $G_{\mu^{k}}(c^*)=G(c^*)$, so they are optimal by Prop.~\ref{prop:classical}(c).
\end{proof}

The PI algorithm of Eqs. \eqref{eq:pi_eva_pos}, \eqref{eq:pi_imp_pos} also admits an optimistic variant. For the semilinear problem studied here, it can be described in terms of vectors $c_k$ of linear cost functions $J_k(x)=c_k'x$ and associated policies $\mu^k$. In particular, let $\{\ell_k\}$ be any sequence of positive integers chosen as design parameters. The optimistic PI algorithm starts with some $c_0$ such that 
\begin{equation}
    \label{eq:op_pi_initial_pos_dete}
    c_0\geq G(c_0).
\end{equation}
At a typical iteration $k$, given $c_k$, it computes a policy $\mu^k\in \widehat{\cal M}$ such that
\begin{equation}
    \label{eq:op_pi_policy_dete}
    G_{\mu^k}(c_k)=G(c_k),
\end{equation}
and it obtains $c_{k+1}$ by
\begin{equation}
    \label{eq:op_pi_cost_dete}
    c_{k+1}'=c_k'(\alpha A_{\mu^k})^{\ell_k}+q_{\mu^k}'\sum_{i=0}^{\ell_k-1}(\alpha A_{\mu^k})^i.
\end{equation}
Note that there exists a vector $c_0\in \Re^n_+$ that satisfies the inequality \eqref{eq:op_pi_initial_pos_dete} in view of Prop.~\ref{prop:unique_dete}. In what follows, we will show that the sequence $c_k$ generated by optimistic PI converges to $c^*$, and the cost vectors associated with $\mu^k$ also converge to $c^*$.

\begin{proposition}
    \label{prop:optimistic_pi_dete}
    Let $\{\m^k\}$ be a sequence generated by the optimistic PI algorithm  \eqref{eq:op_pi_policy_dete}-\eqref{eq:op_pi_cost_dete}. Then for every $k$, $\mu^k\in \widehat{\cal M}$, $\alpha A_{\mu^k}$ is stable, and for some $c_{\mu^k}\in \re_+^n$, $J_{\mu^k}(x)=(c_{\mu^k})'x$ for all $x\in X$. Moreover, we have $c_k\to c^*$ and $c_{\mu^k}\to c^*$ as $k\to\infty$.
\end{proposition}

\begin{proof}
   The stability of the matrices $\alpha A_{\mu^k}$ can be shown by using arguments similar to those in the proof for Prop.~\ref{prop:pi_dete}. As a result, the cost function $J_{\mu^k}$ of $\mu^k$ can be written as $J_{\mu^k}(x)=(c_{\mu^k})'x$ for all $x\in X$, where
   $$c_{\mu^k}'=q_{\mu^k}'(I-\alpha A_{\mu^k})^{-1},$$
   cf. Eq.~\eqref{eq:pi_eva_pos}.

  To prove the convergence of optimistic PI, we consider an auxiliary sequence $\{\overline{c}_k\}$ that is generated with the iteration $\overline{c}_{k+1}=G(\overline{c}_k)$, starting with $\overline{c}_0=c_0$. The sequences $\{c_k\}$ and $\{\overline{c}_k\}$ define sequences of functions $\{J_k\}$ and $\overline{J}_k$ via $J_k(x)=(c_k)'x$ and $\overline{J}_k(x)=(\overline{c}_k)'x$, respectively. In what follows, we will show by induction the inequalities 
    \begin{align}
             &G(\overline{c}_k)\leq \overline{c}_k,\qquad k=0,1,\dots,\label{eq:op_pi_inequality_c}\\
             &G(c_k)\leq c_k,\qquad k=0,1,\dots,\label{eq:op_pi_inequality_b}\\
            &c^*\leq c_k\leq \overline{c}_k,\qquad k=0,1,\dots,\label{eq:op_pi_inequality_a}
    \end{align}
    one after the other. 

   Starting with $k=0$, from the definition of $c_0$, we have that $c_0= \overline{c}_0$ and $G(c_0)\leq c_0$. As a result, $G(c_0)'x\leq c_0'x$ for all $x\in X$. By Prop.~\ref{prop:classical}(a), this implies that $J^*(x)\leq c_0'x$ for all $x\in X$. Since $J^*(x)=(c^*)'x$, we have $c^*\leq c_0$. Therefore, Eqs.~\eqref{eq:op_pi_inequality_c}, \eqref{eq:op_pi_inequality_b}, and \eqref{eq:op_pi_inequality_a} hold for $k=0$. 
   
   Let us assume that Eqs.~\eqref{eq:op_pi_inequality_c}, \eqref{eq:op_pi_inequality_b}, and \eqref{eq:op_pi_inequality_a} hold for $k$. We will show that $\overline{c}_{k+1}$ and $c_{k+1}$ satisfies Eqs.~\eqref{eq:op_pi_inequality_c} and \eqref{eq:op_pi_inequality_b}, respectively. The inequality $c^*\leq c_{k+1}\leq \overline{c}_{k+1}$ will be shown afterwards.
    
    By the monotonicity of $G$, the inequality $G(\overline{c}_k)\leq \overline{c}_k$ implies that $G^2(\overline{c}_k)\leq G(\overline{c}_k)$. Since $\overline c_{k+1}=G(\overline{c}_k)$, we have that $\overline{c}_{k+1}$ satisfies Eq.~\eqref{eq:op_pi_inequality_c}. Next, we show that $c_{k+1}$ satisfies Eq.~\eqref{eq:op_pi_inequality_b}. By the definition of $\mu^{k}$, we have that $G_{\mu^{k}}(c_k)=G(c_k)\leq c_k$. Since $G_{\mu^{k}}$ is monotone, applying $G_{\mu^{k}}$ multiple times on both sides of $G_{\mu^{k}}(c_k)\leq c_k$ preserves the inequality. This yields
    \begin{equation}
    \label{eq:g_m_kp1}
        G_{\mu^{k}}^{\ell_k+1}(c_k)\leq G_{\mu^{k}}^{\ell_k}(c_k)\leq G_{\mu^{k}}(c_k)
    \end{equation}
    From Eq.~\eqref{eq:op_pi_cost_dete}, we have that $c_{k+1}=G_{\m^{k+1}}^{\ell_k}(c_k)$. As a result, the first inequality in Eq.~\eqref{eq:g_m_kp1} can be written as
    $$G_{\mu^{k}}(c_{k+1})\leq c_{k+1}.$$
    Given that $G(c_{k+1})\leq G_{\mu^{k}}(c_{k+1})$, [cf. Eq.~\eqref{eq:g_le_gm}], we have $G(c_{k+1})\leq c_{k+1}$.

    Next, we show that Eq.~\eqref{eq:op_pi_inequality_a} holds for $k+1$. Applying similar arguments to the ones used to show $c^*\leq c_0$, it follows that $G(c_{k+1})\leq c_{k+1}$ implies $c^*\leq c_{k+1}$. To show the last remaining inequality, $c_{k+1}\leq \overline{c}_{k+1}$, we apply $G$ to  both sides of the inequality $c_k\leq \overline{c}_k$. We obtain 
    \begin{equation}
    \label{eq:g_mk_l}
        G_{\mu^{k}}(c_k)=G(c_k)\leq G(\overline{c}_k)=\overline{c}_{k+1},
    \end{equation}
    where the first and the second equalities follow from the definitions of $\mu^{k}$ and $\overline{c}_{k+1}$, respectively. Combining the second inequality in Eq.~\eqref{eq:g_m_kp1} with Eq.~\eqref{eq:g_mk_l} yields
    $$c_{k+1}=G_{\mu^{k}}^{\ell_k}(c_k)\leq G_{\mu^{k}}(c_k)\leq\overline{c}_{k+1}.$$
    This concludes the induction proof of Eqs.~\eqref{eq:op_pi_inequality_c}-\eqref{eq:op_pi_inequality_a}.
    
    To prove the convergence of $\{c_k\}$, we note that by Prop.~\ref{prop:vi_dete}, we have $\overline{c}_k\to c^*$. In view of Eq.~\eqref{eq:op_pi_inequality_a}, this implies $c_k\to c^*$. By the definition of $\mu^k$, we have that $G_{\mu^{k}}(c_k)=G(c_k)\leq c_k$, which implies $c_{\mu^k}\leq c_k$; cf. Prop.~\ref{prop:classical}(b). In addition, we also have $c^*\leq c_{\mu^{k}}$. As a result, $c_k\to c^*$ implies $c_{\mu^{k}}\to c^*$.
\end{proof}

\subsection{A Computational Approach Based on Mathematical Programming}
A well-known computational method in infinite horizon DP is based on solving a mathematical programming problem, where the constraint 
\begin{equation}
    \label{eq:prog_constraint}
    J(x)\leq\min_{u\in U(x)}\Big[g(x,u)+\alpha J\big(f(x,u)\big)\Big],\qquad\text{for all }x,
\end{equation}
is imposed on the functions $J$. In this section, we will show how this method can be applied to our semilinear DP problem, in a way that takes advantage of its structure. In particular, we will formulate a mathematical programming problem in the space of the $n$-dimensional parameters $c$, rather than in the infinite-dimensional space of cost functions $J:X\mapsto \re$. This approach has already been applied in the paper by Rantzer \cite{rantzer2022explicit} to the positive linear problem discussed in Example~\ref{eg:pos_lin} (see also the paper by Li and Rantzer \cite{li2024exact}), and it will be generalized to our semilinear DP problem in this section.

Similar to the VI and PI algorithms discussed earlier, we express the mathematical programming problem in terms of the parameter vector as follows:
\begin{subequations}
\label{eq:op_dete}
    \begin{align}
	\max_{c}& \quad \sum_{i=1}^nc^i\\
	\mathrm{s.\,t.} & \quad c\leq G(c),\label{eq:op_bellman_dete}\\
	& \quad  c\in \Re^n_+,\label{eq:op_range_dete}
	\end{align}
\end{subequations} 
where $c^i$ denotes the $i$th component of $c$. Note that this problem is convex, because $G$ is a concave function, since it is defined as the minimum of linear functions. We next show that the maximum is attained at the parameter  $c^*$ of the optimal cost function, i.e., $J^*(x)=(c^*)'x$.

\begin{proposition}
    \label{prop:op_dete}
    The optimal parameter vector $c^*$ is the unique optimal solution of the mathematical programming problem \eqref{eq:op_dete}. 
\end{proposition}

\begin{proof}
    Let  $c$ be a  feasible vector and consider the sequence $\{c_k\}$ generated by $c_{k+1}=G(c_k)$ with $c_0=c$. By Prop.~\ref{prop:vi_dete}, $\{c_k\}$ is monotonically increasing and converges to $c^*$, where $c^*$ is the unique nonnegative vector satisfying $c^*=G(c^*)$, and $J^*(x)=(c^*)'x$ for all $x\in X$. Moreover, $c^*$ is a feasible vector. Therefore, the maximum is attained at $c^*$. The uniqueness assertion follows from Prop.~\ref{prop:unique_dete}.
\end{proof}

\section{Stochastic Positive Semilinear Problems and Certainty Equivalence}\label{sec:stochastic}
Let us consider a stochastic extension of the problem of Section~\ref{sec:dete_formulation}. We introduce a set $\Theta$ of parameters,  whose elements are generically denoted by $\theta$. At each stage $k$, a parameter $\theta_{k}$ is generated according to a known stationary distribution, independently of the preceding parameters $\theta_0,\ldots,\theta_{k-1}$. The parameter $\theta_{k}$ affects  the evolution of both the state and the cost at stage $k$.\footnote{We do not address measurability issues in our subsequent problem formulation. Alternatively, we can assume that $\Theta$ is a countable set, in which case measurability is of no essential concern.} 

Suppose that $X$ is a given subset of $\Re_+^n$. The stochastic problem is defined as follows: for every state $x_0\in X$, solve
\begin{equation}
    \label{eq:proq_stoch}
    \begin{aligned}
        \min_{\{\mu_k\}_{k=0}^{\infty}}& \qquad \lim_{N\to\infty}\mathop{E}_{\substack{\theta_k \\ k=0,\dots,N-1}}\Bigg\{\sum_{k=0}^{N-1} \alpha^k g\big(x_k,\mu_k(x_k),\theta_{k}\big)\Bigg\}\\
	\mathrm{s.\,t.} & \qquad x_{k+1}=f\big(x_k,\mu_k(x_k),\theta_{k}\big),\quad k = 0,1,...,\\
	& \qquad \mu_k(x_k)\in U(x_k),\quad k = 0,1,...,
    \end{aligned}
\end{equation}
where $U(x_k)$ are nonempty control constraint sets. The functions $\mu_k$, $k=0,1,\dots$, are policies, i.e., they map states $x_k$ to elements in $U(x_k)$. We denote by $J^*(x_0)$ the optimal cost of problem \eqref{eq:proq_stoch} starting from $x_0$.

Similar to the deterministic problem considered earlier, we assume that for all $x\in X$, $u\in U(x)$, $\theta\in \Theta$, the functions $f$ and $g$ satisfy
\begin{equation}
    \label{eq:state_stoch}
    f(x,u,\theta)\in X,\quad E_\theta\big\{f(x,u,\theta)\big\}\in X,\qquad g(x,u,\theta)\ge0.
\end{equation}
We focus on the set of nonnegative linear cost functions $\widehat{\cal J}$ and a corresponding set of policies $\widehat {\cal M}$. In particular, we require that for every $\mu\in \widehat{\cal M}$ and $\theta\in \Theta$, there exists an $n\times n$ nonnegative matrix $A_\m^\theta$, and an $n$-dimensional vector $q_\m^\theta$ such that
\begin{equation}
    \label{eq:semil_cond_stoch}
    f\big(x,\m(x),\theta\big)=A_\m^\theta x,\qquad g\big(x,\m(x),\theta\big)=(q_\m^\theta)'x.
\end{equation}

The stochastic terms involving $\theta$ in Eq.\ \eqref{eq:semil_cond_stoch} are usually  referred to as \emph{multiplicative} (as opposed to additive) in the literature. Linear-quadratic problems with multiplicative random terms were studied by Wonham \cite{wonham1967optimal}, with subsequent works in economics and engineering including \cite{chow1973effect,speyer1974stochastic,athans1977uncertainty,ku1977further,rami2002discrete,zhang2015linear}. Some recent works are \cite{gravell2020learning,xing2022identification,pang2022robust,granzotto2024stability}.

For our stochastic problem \eqref{eq:proq_stoch}, we make the following standing assumption, which is patterned after the deterministic assumption of Section~\ref{sec:dete_formulation}.

\begin{assumption}\label{asm:stoch_well}
\begin{itemize}
    \item[(a)] \emph{Closure and Attainability}: The set of nonnegative linear functions $\widehat{\cal J}$ is closed under VI in the sense that 
for every $c\in\re^n_+$, the function
$$\min_{u\in U(x)}E_\theta\Big\{g(x,u,\theta)+\alpha c'f(x,u,\theta)\Big\}$$
belongs to  $\widehat{\cal J}$, i.e., it has the form $\hat c'x$ for some unique $\hat c\ge0$. Furthermore, $\hat c$ depends continuously on $c$. Moreover, there is a policy $\m\in\widehat {\cal M}$ that attains the minimum above, in the sense that
$$\m(x)\in\arg\min_{u\in U(x)}E_\theta\Big\{g(x,u,\theta)+\alpha c'f(x,u,\theta)\Big\},\quad \text{for all }x\in X.$$
\item[(b)] \emph{Stabilizability}: There exists a policy $\hat{\mu}\in \widehat {\cal M}$ such that $\alpha E_\theta\{A_{\hat{\mu}}^\theta\}$ is stable.
\item[(c)] \emph{State Space Structure}: The state space $X$ has the property that for all $v\in \Re_+^n$, there exists an $x\in X$ and a scalar $s\geq0$ such that $sx= v$. 
\item[(d)] \emph{Observability}: There exists an integer $N$ such that the optimal cost of the $N$-stage version of the problem [the problem of minimizing the cost 
$$\mathop{E}_{\substack{\theta_k \\k=0,\dots,N-1}}\Bigg\{\sum_{k=0}^{N-1}\alpha^kg(x_k,u_k,\theta_k)\Bigg\},$$
starting from any nonzero initial state $x_0\in X$] is strictly positive.
 \end{itemize}
\end{assumption}

Before considering the solution to problem \eqref{eq:proq_stoch}, let us provide an example to illustrate Assumption~\ref{asm:stoch_well}(a). Other examples that involve bilinear systems with stochastic parameters can be found in \cite[Exercises 14, 15, p.~68]{bertsekas1976dynamic}. See also \cite[Exercises~1.21, 1.22]{bertsekas2017dynamic}.

\begin{example}[Stochastic Positive Linear System]\label{eg:pos_lin_stoch} Let us consider the stochastic version of the problem of Example~\ref{eg:pos_lin}. Given a state $x_k$ and control $u_k$, a random parameter $\theta_k$ that belongs to a given set $\Theta$ is generated according to a given distribution, and is independent from the parameters of previous stages, and $(x_k,u_k)$. The system equation is given by
$$x_{k+1}=A^{\theta_k}x_k+B^{\theta_k}u_k,$$
where for each $\theta\in \Theta$, $A^\theta$ and $B^\theta$ are $n\times n$ and $n\times m$ matrices, respectively.\footnote{Without loss of generality, we assume both $A^\theta$ and $B^\theta$ depend on the same random variable $\theta$. If they depend on different random quantities $\xi$ and $\zeta$, we can redefine $\theta=(\xi,\zeta)$.} The $k$th stage cost for applying $u_k$ at $x_k$ with random parameter $\theta_k$ is given by
$$q_{\theta_k}'x_k+r_{\theta_k}'u_k,$$
where for every $\theta$, $q_{\theta}\in \re^n_+$ and $r_{\theta}\in \re^m$. The control $u$ at  state $x$ is selected from the set $U(x)=\{u\in\re^m\,|\,|u|\leq Hx\}$, where the matrix $H$ is known. As in Example~\ref{eg:pos_lin}, under suitable conditions we have that starting with $x_0\in \re_+^n$, all the subsequent states $x_k$, $k\ge 1$, remain in $\re_+^n$, regardless of the choices of $u_k$.

We consider the same set of linear policies $\widehat{\cal M}$ as in Example~\ref{eg:pos_lin}. Let us apply VI starting with $J(x)=c'x$, where $c\in \re_+^n$. It produces the function
\begin{equation}
\label{eq:bellman_pos_lin_st}
    \hat{J}(x)=(q+A'c)x+\min_{|u|\leq Hx}(r+B'c)'u,
\end{equation}
where $q=E_\theta\big\{q_\theta\big\}$, $r=E_\theta\big\{r_\theta\big\}$, $A=E_\theta\big\{A^\theta\big\}$, and $B=E_\theta\big\{B^\theta\big\}$. 
It can be seen that Eq.~\eqref{eq:bellman_pos_lin_st} is identical to Eq.~\eqref{eq:bellman_pos_lin}. As a result, we have $\hat{J}(x)=\hat{c}'x$, where $\hat{c}$ is computed using the formula given in Example~\ref{eg:pos_lin}. Moreover, the policy $\m$ that minimizes the right side of Eq.~\eqref{eq:bellman_pos_lin_st} is identical to the policy of the deterministic problem in Example~\ref{eg:pos_lin}. 
\end{example}

In the preceding example, we verified Assumption~\ref{asm:stoch_well}(a) by demonstrating a close relation between the stochastic problem and its deterministic counterpart: the VI algorithm produces identical iterates for both the stochastic and the deterministic problems when starting with the same linear function, and the policies that attain the minima in the VI calculation are the same. These are manifestations of the well-known \emph{certainty equivalence principle} (CEP for short) of this stochastic problem. The meaning of CEP is that the stochastic problem can be solved by addressing a deterministic problem, which is obtained by setting all random variables to their expected values. For linear quadratic problems with additive noise, CEP was first discussed by Simon \cite{simon1956dynamic}, and was also discussed by Wonham for linear-quadratic problems with  multiplicative noise \cite{wonham1967optimal}.  An example of CEP for semilinear finite horizon problems was given in Section~\ref{sec:intro}. We will establish CEP for problem \eqref{eq:proq_stoch}. As a result, the computational approaches developed in Section~\ref{sec:algo_dete} can be brought to bear. 

Let us introduce some notation. For every policy $\mu\in \widehat{\cal M}$, we define an $n\times n$ matrix $A_\mu$ and an $n$-dimensional vector $q_\mu$ as follows:
\begin{equation}
    \label{eq:semi_lin_stoch}
    A_\mu=E_\theta\Big\{A_\mu^\theta\Big\},\qquad q_\mu=E_\theta\Big\{q_\mu^\theta\Big\}.
\end{equation}
Since $A_\m^\theta$ and $q_\m^\theta$ are nonnegative for all $\theta$, it follows that $A_\mu$ and $q_\m$ are nonnegative. Moreover, combining Eqs.~\eqref{eq:semil_cond_stoch} and \eqref{eq:semi_lin_stoch}, we have that for every $\mu\in \widehat{\cal M}$
$$E_\theta\Big\{f\big(x,\m(x),\theta\big)\Big\}=A_\m x,\quad E_\theta\Big\{g\big(x,\m(x),\theta\big)\Big\}=(q_\m)'x,\qquad \text{for all }x.$$
For every $\mu\in \widehat{\cal M}$, let us consider the mapping $\widehat{G}_\m:\re_+^n\mapsto\re_+^n$, given by
$$\widehat{G}_\m(c)=q_\m+\alpha A_\m'c.$$
In view of Assumption~\ref{asm:stoch_well}(a), the VI algorithm, applied to functions $J\in\widehat{\cal J}$ of the form $J(x)=c'x$, defines uniquely a function $\widehat{G}:\re_+^n\mapsto\re_+^n$ through the equation
$$\widehat{G}(c)'x=\min_{u\in U(x)}E_\theta\big\{g(x,u,\theta)+\alpha c'f(x,u,\theta)\big\}=\min_{\mu\in\widehat {\cal M}}{\widehat{G}_{\mu}(c)'x},\quad \text{for all }x\in X.$$

Let us now address the stochastic problem \eqref{eq:proq_stoch} by considering its deterministic counterpart. To this end, we introduce functions $\hat{f}$ and $\hat{g}$ defined as
\begin{equation}
    \label{eq:mean_nl_dy}
    \hat{f}(x,u)=E_\theta\big\{f(x,u,\theta)\big\},\qquad  \hat{g}(x,u)=E_\theta\big\{g(x,u,\theta)\big\}.
\end{equation}
Similar to Section~\ref{sec:dete_formulation}, we consider the deterministic problem involving $\hat{f}$ and $\hat{g}$:
\begin{equation}
    \label{eq:proq_stoch_equiv}
    \begin{aligned}
        \min_{\{u_k\}_{k=0}^{\infty}}& \quad \sum_{k=0}^{\infty} \alpha^k\hat{g}(x_k,u_k)\\
	\mathrm{s.\,t.} & \quad x_{k+1}=\hat{f}(x_k,u_k),\ \ k = 0,1,...,\\
	& \quad u_k\in U(x_k),\ \ k = 0,1,...,
    \end{aligned}
\end{equation}
where $x_0\in X$ is given. Let us denote by $\hat{J}^*(x_0)$ the optimal cost starting from $x_0$. The following proposition shows that it suffices to solve the deterministic problem \eqref{eq:proq_stoch_equiv}.

\begin{proposition}[Certainty Equivalence]\label{prop:stoch_ce}

    \begin{itemize}
        \item[(a)] There exists a vector $c^*\in\re^n_+$ such that $J^*(x)=\hat{J}^*(x)=(c^*)'x$ for all $x\in X$. Moreover, $c^*$ is the unique vector within $\Re^n_+$ that satisfies
        $$c^*=\widehat{G}(c^*).$$
        \item[(b)] A policy $\mu\in \widehat{\cal M}$ is optimal for the deterministic problem  \eqref{eq:proq_stoch_equiv} if and only if it is optimal for the stochastic problem \eqref{eq:proq_stoch}. Moreover, there exists at least one optimal policy.
    \end{itemize}
\end{proposition}

\begin{proof}
(a) Since $f(x,u,\theta)\in X,\,E_\theta\big\{f(x,u,\theta)\big\}\in X$, and $g$ is nonnegative [cf. Eq.~\eqref{eq:state_stoch}], it can be seen that $\hat{f}(x,u)\in X$, $\hat{g}(x,u)\geq 0$ for all $x\in X$ and $u\in U(x)$. Moreover, under Assumption~\ref{asm:stoch_well}, the deterministic problem \eqref{eq:proq_stoch_equiv} satisfies Assumption~\ref{asm:dete_well}, with $\hat{f}$ and $\hat{g}$ given by Eq.~\eqref{eq:mean_nl_dy} in place of $f$ and $g$, and with $A_\mu$, $q_\mu$ defined in Eq.~\eqref{eq:semi_lin_stoch}. Consequently, we can apply the results of Sections~\ref{sec:bellman_dete} and \ref{sec:algo_dete} for problem \eqref{eq:proq_stoch_equiv}. In particular, by Prop.~\ref{prop:op_dete}, there exists a unique vector $c^*\in\Re_+^n$ such that
        $$\hat{J}^*(x)=(c^*)'x,\qquad c^*=G(c^*).$$
Furthermore, the operator $G$ for the deterministic problem \eqref{eq:proq_stoch_equiv} is identical to the operator $\widehat{G}$ for the stochastic problem \eqref{eq:proq_stoch}. As a result, $\hat{J}$ satisfies the Bellman equation for the stochastic problem. By Prop.~\ref{prop:classical}(a), we have that $J^*(x)\leq \hat{J}^*(x)$ for all $x\in X$.

Conversely, let us consider the sequence of functions $\{J_k\}$ defined by $J_k(x)=x'c_k$, where $c_0=0$, and $c_{k+1}=\widehat{G}(c_k)$. Since $\widehat{G}$ and $G$ are identical, we have that $c_k\to c^*$ according to Prop.~\ref{prop:vi_dete}. Moreover, similar to the proof of Prop.~\ref{prop:unique_dete}, we can show that  $J_k(x)\leq J^*(x)$. Hence, $\hat{J^*}(x)\leq J^*(x)$. Together with the earlier inequality $J^*(x)\leq \hat{J}^*(x)$, we conclude that $J^*= \hat{J}^*$.
        
(b) From Prop.~\ref{prop:classical}(c) in Appendix~\ref{app:a}, a policy $\mu^*\in \widehat{\cal M}$ is optimal for problem \eqref{eq:proq_stoch_equiv} if and only if $G_{\mu^*}(c^*)=G(c^*)$. Because the operators $G$ and $\widehat{G}$ are identical, and $J^*(x)=(c^*)'x$, we also have $\widehat{G}_{\mu^*}(c^*)=\widehat{G}(c^*)$.  Thus, such a policy is also optimal for the stochastic problem. Moreover, by Prop.~\ref{prop:optimal_policy}, there exists an optimal policy for problem \eqref{eq:proq_stoch_equiv}.

\end{proof}

\section{Extension to Markov Jump Problems and Deterministic Equivalence}\label{sec:markov}
In this section, we consider an extension to the stochastic problem of Section~\ref{sec:stochastic}. Here, again there is a state $x_k$ that takes values in a subset $X$ of $\re^n_+$, and as in the preceding section, there is also a parameter $\theta_k$ at stage $k$. However, we now allow the parameters of different stages to be correlated and to evolve according to a known Markov chain. We also assume that the current value of the parameter is known before each control is selected, and that the parameter set is finite. In particular, we assume that $\theta$ takes values in the set $\Theta=\{1,2,\dots,r\}$, and we denote by $p_{ij}$ the probability of $\theta_{k+1}=j$ given that $\theta_{k}=i$.\footnote{The problems studied in Section~\ref{sec:stochastic} are special cases of the present section when the parameter set $\Theta$ is finite. However, Section~\ref{sec:stochastic} also allows $\Theta$ to be countably infinite, a case not covered in the present section.}

\begin{figure}
    \centering
    \includegraphics[width=\linewidth]{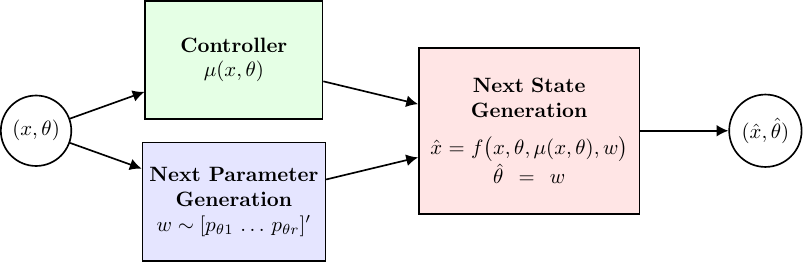}
    \caption{Evolution of the augmented state under a policy $\mu$ in the Markov jump problem. Starting from an augmented state $(x,\theta)$, the controller applies $\mu(x,\theta)$. Meanwhile, the next‐stage parameter $w$ is drawn according to the distribution $[p_{\theta1}\,\dots\,p_{\theta r}]'$. Then, the next state $(\hat x,\hat \theta)$ is generated according to $\hat x= f(x,\theta,\mu(x,\theta),w)$ and $\hat\theta=w$.}
    \label{fig:mj_general}
\end{figure}
 
To capture the dependence between successive parameter values, we introduce  an augmented state 
$(x_k,\theta_k)\in X\times \Theta$, and policies $\m_k(x_k,\theta_k)$ that depend on this augmented state. To this end, we must define a stochastic dynamic system that describes the evolution from the augmented state $(x_k,\theta_k)$ to the next augmented state $(x_{k+1},\theta_{k+1})$. We thus introduce a stochastic variable $w_k$ that models the next value of the parameter, $\theta_{k+1}=w_k$, and evolves according to the given Markov chain transition probabilities; see Fig.~\ref{fig:mj_general}.

Then the problem of this section is defined as follows: for every augmented state $(x_0,\theta_0)\in X\times \Theta$, solve
\begin{equation}
    \label{eq:proq_markov}
    \begin{aligned}
        \min_{\{\mu_k\}_{k=0}^{\infty}}& \qquad \lim_{N\to\infty}\mathop{E}_{\substack{w_k \\k=0,\dots,N-1}}\Bigg\{\sum_{k=0}^{N-1} \alpha^kg\big(x_k,\theta_k,\mu_k(x_k,\theta_k),w_{k}\big)\Bigg\}\\
	\mathrm{s.\,t.} & \qquad x_{k+1}=f\big(x_k,\theta_k,\mu_k(x_k,\theta_k),w_{k}\big),\quad k = 0,1,...,\\
    & \qquad\theta_{k+1} = w_k,\quad k = 0,1,...,\\
	& \qquad \mu_k(x_k,\theta_k)\in U(x_k,\theta_k),\quad k = 0,1,...,
    \end{aligned}
\end{equation}
where for all admissible $x,u,x^i,u^i$, and all $\theta,w\in\Theta$, the function $f$ satisfies 
\begin{equation}
    \label{eq:state_markov}
    f(x,\theta,u,w)\in X,\qquad
    \sum_{i=1}^r p_{iw}f(x^i,i,u^i,w)\in X .
\end{equation}
and the cost per stage $g$ satisfies
\begin{equation}
    \label{eq:cost_markov}
    g(x,\theta,u,w)\geq 0.
\end{equation}

Problems such as the one of Eq.~\eqref{eq:proq_markov} are often referred to as \emph{Markov jump} problems, since the parameter $\theta_k$ evolves according to a given Markov chain. Markov jump problems with a linear state equation and a quadratic cost were first addressed by Krasovsky and Lidskii \cite{krasovsky1961}. Similar problems were studied later by Sworder \cite{sworder1969feedback} and by Wonham \cite{wonham1970random}. Among subsequent papers, we mention \cite{chizeck1986discrete}, \cite{abou1995solution}, and \cite{costa2002weak}. For related monographs, see \cite{mariton1990jump,costa2005discrete}.

Similar to our earlier analysis, the semilinear problem of this section involves a set of functions $\widehat{\cal J}$ and a corresponding subset of stationary policies $\widehat{\cal M}$. In particular, each $J\in \widehat{\cal J}$ has the form $J(x,\theta)=c(\theta)'x$ with $c$ belonging to the set of functions that map $\Theta$ to $\Re_+^n$, which we denote by  $\mathcal{C}$:
$$\mathcal{C}=\big\{c\,|\,c=\big(c(1),\dots,c(r)\big),\ c(\theta)\in \re_+^n\hbox{ for all }\theta=1,\dots,r\big\}.$$

In analogy with earlier sections, we assume that for every $\mu\in \widehat{\cal M}$,  $\theta,\,w\in \Theta$, there exists an $n\times n$ nonnegative matrix $A_\mu^{\theta w}$, and an $n$-dimensional nonnegative vector $q_\mu^{\theta w}$ such that  
\begin{equation}
        \label{eq:dy_markov_dete}
        f\big(x,\theta,\mu(x,\theta),w\big)=A_{\mu}^{\theta w} x,\quad g\Big(x,\theta,\mu(x,\theta),w\Big)=(q_{\mu}^{\theta w})'x,\qquad\hbox{for all }x.
    \end{equation}

We will show that a different type of equivalence holds for the Markov jump problem \eqref{eq:proq_markov}, and that the deterministic equivalent problem has state dimension $rn$, where $r$ is the number of parameters, and $n$ is the dimension of the vector $x$. In particular, we will show that each policy $\mu\in\widehat{\cal M}$ of the deterministic problem is characterized by the $rn\times rn$ matrix $\Bar{A}_\mu$ and the $rn$-dimensional vector $\Bar{q}_\mu$ given by
\begin{equation}
    \label{eq:a_mu_markov_dete}
    \Bar{A}_\mu=\begin{bmatrix}
p_{11}A_\mu^{11} & p_{21}A_\mu^{21} & \cdots & p_{r1}A_\mu^{r1}\\
p_{12}A_\mu^{12} & p_{22}A_\mu^{22} & \cdots & p_{r2}A_\mu^{r2}\\
\vdots  & \vdots  & \ddots & \vdots  \\
p_{1r}A_\mu^{1r} & p_{2r}A_\mu^{2r} & \cdots & p_{rr}A_\mu^{rr}
\end{bmatrix},\qquad \Bar{q}_\mu=\begin{bmatrix}
    E_{w}\Big\{q_{\mu}^{1w}\,\big|\,\theta=1\Big\}\\
    E_{w}\Big\{q_{\mu}^{2w}\,\big|\,\theta=2\Big\}\\
    \vdots\\
    E_{w}\Big\{q_{\mu}^{rw}\,\big|\,\theta=r\Big\}
\end{bmatrix}.
\end{equation}
This parallels the analysis of Section~\ref{sec:dete_formulation}, where each policy $\mu\in\widehat{\cal M}$ was characterized by a matrix $A_{\mu}$ and a vector $q_{\mu}$, cf. Eq. \eqref{eq:semil_cond}. For the Markov jump problem \eqref{eq:proq_markov},  $\Bar{A}_\mu$ and $\Bar{q}_\mu$ will play the same role as the one played by $A_\mu$ and $q_\mu$ in Section~\ref{sec:dete_formulation}.\footnote{Deterministic reformulations also appear in the study of positive Markov jump systems, where the evolution of first- and second-order moments is described by higher-dimensional deterministic linear systems; see, e.g., \cite{bolzern2015positive}. The objective in that literature is to analyze stability and positivity through moment dynamics, which is different from our goal of establishing a semilinear DP equivalence for optimal control. A further context where deterministic reformulations are employed is optimal filter design, a connection that we return to at the end of this section.}

Similar to the deterministic and the stochastic problems discussed in Sections 2 and 5, we make the following standing assumption for this section.
\begin{assumption}\label{asm:markov_well}
\begin{itemize}
    \item[(a)] \emph{Closure and Attainability}: The set of nonnegative functions $\widehat{\cal J}$ is closed under VI in the sense that 
for every $c\in{\cal C}$, the function
$$\min_{u\in U(x,\theta)}E_{w}\Big\{g(x,\theta,u,{w})+\alpha c({w})'f(x,\theta,u,{w})\,\big|\,\theta\Big\}$$
belongs to $\widehat{\cal J}$, i.e., it has the form ${\hat c}(\theta)'x$ for some unique $\hat c\in{\cal C}$. Furthermore, $\hat c$ depends continuously on $c$. Moreover, for every $\theta\in \Theta$, there is a policy $\m\in\widehat {\cal M}$ that attains the minimum above, in the sense that
$$\m(x,\theta)\in\arg\min_{u\in U(x,\theta)} E_{w}\Big\{g(x,\theta,u,{w})+\alpha c({w})'f(x,\theta,u,{w})\,\big|\,\theta\Big\},$$
for all $x\in X$.
\item[(b)] \emph{Stabilizability}: There exists a policy $\hat{\mu}\in \widehat {\cal M}$ such that $\alpha \Bar{A}_{\hat{\mu}}$ is stable.
\item[(c)] \emph{State Space Structure}: The state space $X$ has the property that for some scalar $s>0$, if $x\in\re_+^n$ and $\sum_{i=1}^nx^i\leq s$, where $x^i$ denotes the $i$th component of $x$, then $x\in X$.  
\item[(d)] \emph{Observability}: There exists an integer $N$ such that the optimal cost of the $N$-stage version of the problem [the problem of minimizing the cost 
$$\mathop{E}_{\substack{w_k \\k=0,\dots,N-1}}\Bigg\{\sum_{k=0}^{N-1}\alpha^kg(x_k,\theta_k,u_k,w_{k})\Bigg\},$$
starting from any nonzero initial state $x_0\in X$ and any parameter $\theta_0\in \Theta$] is strictly positive.
\item[(e)] \emph{Cartesian Product Condition}: For every $\mu^1,\dots,\mu^r\in \widehat{\cal M}$, the policy that applies the control $\mu^\theta(x,\theta)$ at an augmented state $(x,\theta)$, $x\in X$, $\theta=1,\ldots,r$, belongs to $\widehat{\cal M}$.
 \end{itemize}
\end{assumption}

Note that part (e) of the preceding assumption has no counterpart in Assumptions~\ref{asm:dete_well} and \ref{asm:stoch_well}. To illustrate this part, let us consider policies $\mu^1,\dots,\mu^r\in \widehat{\cal M}$ and for a given state $x\in X$, the matrix of controls
\begin{equation}
    \begin{bmatrix}
\mu^1(x,1) & \mu^2(x,1) & \cdots & \mu^r(x,1)\\
\mu^1(x,2) & \mu^2(x,2) & \cdots & \mu^r(x,2)\\
\vdots  & \vdots  & \ddots & \vdots  \\
\mu^1(x,r) & \mu^2(x,r) & \cdots & \mu^r(x,r)
\end{bmatrix}.\nonumber
\end{equation}
Then Assumption~\ref{asm:markov_well}(e) requires that the policy obtained from the diagonal of this matrix [the one that applies $\mu^\theta(x,\theta)$ at augmented state $(x,\theta)$] belongs to $\widehat{\cal M}$. 

Let us illustrate the problem and Assumption~\ref{asm:markov_well} through the Markov jump extension of Example~\ref{eg:pos_lin}.

\begin{example}[Markov Jump Positive Linear Systems]
    We consider a Markov jump problem involving a positive linear system, and to simplify the presentation, we assume that there are just two parameters: $\Theta=\{1,2\}$. Suppose that the current state is $(x_k,\theta_k)$. Then the state equation is given by 
    $$x_{k+1}=A^{\theta_k}x_k+B^{\theta_k}u_k,$$
    where $A^{\theta}$ and $B^{\theta}$, $\theta=1,2$, are given $n\times n$ and $n\times m$ matrices, respectively. The parameter $\theta_{k+1}$ is generated according to the transition probabilities $p_{\theta_k\theta_{k+1}}$. The cost per stage and the control constraint sets are the same as in Example~\ref{eg:pos_lin}:
    $$q'x_k+r'u_k,\quad U(x)=\{u\in\re^m\,|\,|u|\leq Hx\},$$
    where $q\in\re^n$, $r\in\re^m$, and $H$ is an $m\times n$ nonnegative matrix.

\begin{figure}
    \centering
    \includegraphics[width=\linewidth]{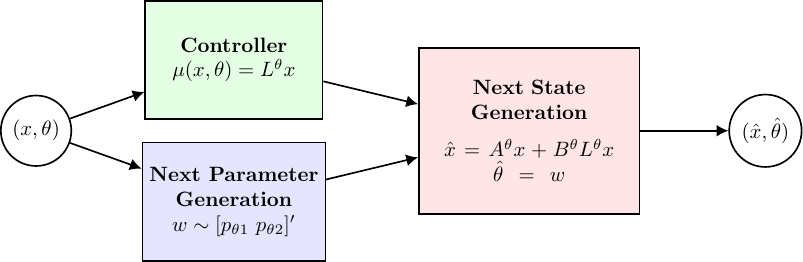}
    \caption{Evolution of the augmented state under a policy $\mu\in\widehat{\cal M}$ in a positive linear system example. Starting from an augmented state $(x,\theta)$, {the controller applies a linear policy $\mu(x,\theta)=L^\theta x$, where the feedback matrix $L^\theta$ depends on the parameter $\theta$.} Meanwhile, the next‐stage parameter $w$ is drawn according to the {two-dimensional distribution} $[p_{\theta1}\;p_{\theta 2}]'$. Then, the next state $(\hat x,\hat \theta)$ is determined by $\hat x= A^\theta x+B^\theta L^\theta x$ and $\hat\theta=w$.}
    \label{fig:mj_linear}
\end{figure}
    
    For this problem, the policies in the set $\widehat{\cal M}$ take the form 
    $$\mu(x,1)=L^1x,\qquad\mu(x,2)=L^2x,$$
    where $L^\theta$, $\theta=1,2$, are $m\times n$ matrices, and $|L^\theta|\leq H$ for $\theta=1,2$. In words, when the controller observes a parameter value $\theta$, it applies the linear policy with feedback gain matrix $L^\theta$. It then keeps applying that policy, up until it detects a parameter change from $\theta$ to $\hat \theta$, in which case it switches  to the feedback gain matrix $L^{\hat\theta}$; cf. Fig.~\ref{fig:mj_linear}

    Through calculations similar to those in Examples~\ref{eg:pos_lin} and \ref{eg:pos_lin_stoch}, we can verify that part (a)-(d) of Assumption~\ref{asm:markov_well} holds. 
    To verify that part (e) holds, let $\mu^1,\mu^2\in \widehat{\cal M}$ so that $$\mu^1(x,1)=L_1^1x,\quad \mu^1(x,2)=L^2_1x,\qquad\mu^2(x,1)=L^1_2x,\quad \mu^2(x,2)=L^2_2x,$$
    where $L_i^j$, $i,j=1,2$, are $n\times m$ matrices, and $|L_i^j|\leq H$, $i,j=1,2$. Then a policy $\mu$ defined as $$\mu(x,1)=L_1^1x,\qquad \mu(x,2)=L_2^2x,$$ also belongs to the set $\widehat{\cal M}$, thus showing that Assumption~\ref{asm:markov_well}(e) is satisfied.
\end{example}

In analogy to Section~\ref{sec:dete_formulation}, for every $\mu\in\widehat{\cal M}$, let us define a function $\overline{G}_\mu:\mathcal{C}\mapsto\mathcal{C}$, where for every $\theta\in \Theta$ and $c\in\mathcal{C}$,
    $$\big(\overline{G}_{\mu}(c)\big)(\theta)=E_{w}\Big\{q_{\mu}^{\theta w}\,\big|\,\theta\Big\}+\alpha E_{w}\Big\{(A_{\mu}^{\theta w})'c(w)\,\big|\,\theta\Big\}.$$
Then Assumption~\ref{asm:markov_well}(a) implies that for all $\theta$ and $c$,
$$\min_{\mu\in \widehat{\cal M}}\big(\overline{G}_{\mu}(c)\big)(\theta)'x=\min_{u\in U(x,\theta)}E_{w}\Big\{g(x,\theta,u,{w})+\alpha c({w})'f(x,\theta,u,{w})\,\big|\,\theta\Big\},\ \ \hbox{for all }x\in X.$$
Similar to Section~\ref{sec:dete_formulation}, we introduce a function $\big(\overline{G}(c)\big)(\theta)$, which for every $\theta$ and $c$, is uniquely defined by the equation 
\begin{equation}
        \label{eq:markov_g_bar}
        \big(\overline{G}(c)\big)(\theta)'x=\min_{\mu\in\widehat{\cal M}}{\big(\overline{G}_\mu(c)\big)(\theta)'x},\qquad \text{for all }x\in X.
    \end{equation}
From Assumption~\ref{asm:markov_well}(a), it follows that the function $\overline{G}$ is continuous.

\subsection*{The Deterministic Problem}

We will now construct a deterministic problem that falls within the framework of Section~\ref{sec:dete_formulation} and is closely related to the Markov jump problem of this section. To this end, we consider the set $\overline{X}$ that consists of $r$-tuples of the form $(x^1,\dots,x^r)$, where $x^i\in X$ for all $i=1,\ldots,r$. These $r$-tuples will serve as the states of the deterministic problem, and will be generically denoted by $\bar x$, i.e., $\Bar{x}=(x^1,\dots,x^r)$. In particular, the state of the deterministic problem at stage $k$ will be denoted by $\bar x_k$ and its components will be denoted by $x^i_k$ for $i=1,\ldots,r$.

The control constraint set at state $\bar x=(x^1,\dots,x^r)$ of the deterministic problem, denoted by $\overline{U}(\bar x)$, is the Cartesian product
\begin{equation}
    \label{eq:markov_u_bar_def}
    \overline{U}(\bar x)=\prod_{\theta=1}^rU(x^\theta,\theta).
\end{equation}
It can be seen that elements in $\overline{U}(\bar x)$ are $r$-tuples of the form $(u^1,\dots,u^r)$, where $u^i\in U(x^i,i)$ for all $i=1,\dots,r$. Similar to the notation for states, these $r$-tuples will be generically denoted by $\bar u$, i.e., $\bar u=(u^1,\dots,u^r)$. In particular, the control of the deterministic problem at stage $k$ will be denoted by $\bar u_k$ and its components will be denoted by $\bar u_k^i$ for $i=1,\dots,r$.

A policy $\mu$ of the Markov jump problem \eqref{eq:proq_markov} defines a policy $\bar \mu$ for the deterministic problem, which maps a state $\bar x= (x^1,\dots,x^r)$ to the control
$\bar \mu(\bar x)$ given by 
\begin{equation}
    \label{eq:markov_extend_mu}
    \bar \mu(\bar x)=\big(\mu(x^1,1),\dots,\mu(x^r,r)\big)
\end{equation}
for the deterministic problem. This is a special type of policy, whereby its $i$th  component $\mu(x^i,i)$ depends on a single corresponding component of the state $\bar{x}$, namely $x^i$, rather than on the entire state vector $\bar x=(x^1,\ldots,x^r)$.

Let us now introduce a state equation and cost per stage for the deterministic equivalent problem. We define the state equation as
\begin{equation}
\label{eq:markov_extend_dy}
    \bar x_{k+1}=\Bar{f}(\bar x_k,\bar u_k)=\big(\Bar{f}^1(\bar x_k,\bar u_k),\dots,\Bar{f}^r(\bar x_k,\bar u_k)\big),
\end{equation}
where
\begin{equation}
    \label{eq:mean_nl_dy_markov}
\Bar{f}^\theta(\bar x,\bar u)=\sum_{i=1}^rp_{i\theta}f(x^i,i,u^i,\theta),\quad \hbox{for all }\bar x\in \overline{X},\ \bar u\in\overline{U}(\bar x),\ \theta=1,\ldots,r,
\end{equation}
and we define a corresponding $k$th stage cost $\Bar{g}(\bar x_k,\bar u_k)$ as
\begin{equation}
    \label{eq:markov_extend_cost}
    \Bar{g}(\bar x_k,\bar u_k)=\sum_{\theta=1}^rE_{w}\Big\{g(x_k^\theta,\theta,u_k^\theta,w)\,\big|\,\theta\Big\}.
\end{equation}

The deterministic  equivalent problem is defined as follows: For every $\bar x_0\in \overline{X}$, solve
\begin{equation}
    \label{eq:proq_markov_equiv}
    \begin{aligned}
        \min_{\{\bar u_k\}_{k=0}^{\infty}}& \qquad \sum_{k=0}^{\infty} \alpha^k\Bar{g}(\bar x_k,\bar u_k)\\
	\mathrm{s.\,t.} & \qquad \bar x_{k+1}=\Bar{f}(\bar x_k,\bar u_k),\quad k = 0,1,...,\\
	& \qquad \bar u_k\in \overline{U}(\bar x_k),\quad k = 0,1,...,
    \end{aligned}
\end{equation}
and we denote by $\overline{J}^*(\bar x_0)$ the optimal cost starting from $\bar x_0$. 

\subsubsection*{Notational Convention}
\begin{itemize}
    \item[(a)] For any policy $\m\in\widehat{\cal M}$ of the Markov jump problem \eqref{eq:proq_markov},  we  denote by $\bar\mu$ the policy \eqref{eq:markov_extend_mu} of the deterministic problem \eqref{eq:proq_markov_equiv}. Moreover we denote by $\overline{{\cal M}}$ the set of policies $\bar \mu$ of the deterministic problem  that are given by Eq.~\eqref{eq:markov_extend_mu} for some $\m\in\widehat{\cal M}$.
    \item[(b)] We denote by $\overline{\cal J}$ the set of cost functions of the deterministic problem \eqref{eq:proq_markov_equiv} of the form $\overline{J}(\bar x)=c'\bar x$, where $\bar x=(x^1,\dots,x^r)$, $c=\big(c(1),\dots,c(r)\big)$, and $c(i)\in \re_+^n$, $i=1,\dots,r$.
\end{itemize} 

Next, we show that the deterministic problem \eqref{eq:proq_markov_equiv} has the semilinear structure introduced in Section~\ref{sec:dete_formulation}: under every policy $\bar \mu\in \overline{\cal M}$, the state equation $\bar f$ and cost per stage $\bar g$ are both linear in $\bar x$; cf. Eq.~\eqref{eq:semil_cond}. Moreover, the deterministic problem satisfies Assumption~\ref{asm:dete_well}, and hence it belongs to the class of problems studied in Sections~\ref{sec:dete_formulation}-\ref{sec:algo_dete}.

\begin{proposition}
\label{prop:markov_equiv_prop}
For every policy $\bar \mu\in \overline{\cal M}$ that corresponds to a policy $\mu\in\widehat{\cal M}$ as per Eq.~\eqref{eq:markov_extend_mu}, the state equation $\bar f$ and cost per stage $\bar g$ of the deterministic problem \eqref{eq:proq_markov_equiv} satisfy
$$\bar f\big(\bar x,\bar\m(\bar x)\big)=\bar A_\m \bar x,\quad \bar g\big(\bar x,\bar\m(\bar x)\big)=\bar q_\m'x,\qquad \hbox{for all }\bar x\in \overline{X},$$
where $\bar A_\m$ and $\bar q_\m$ are given by Eq.~\eqref{eq:a_mu_markov_dete}. Moreover, the deterministic problem  satisfies Assumption~\ref{asm:dete_well} with $f$, $g$, $\widehat{\cal M}$, and $\widehat{\cal J}$ replaced by $\bar f$, $\bar g$, $\overline{\cal M}$, and $\overline{\cal J}$, respectively. 
\end{proposition}

The proof of the proposition is deferred to Appendix~\ref{app:proof}. Since the deterministic problem \eqref{eq:proq_markov_equiv} fits into the semilinear framework of Section~\ref{sec:dete_formulation}, it can be solved with the DP methods developed in Section~\ref{sec:algo_dete}.

\subsection*{Relation Between the Deterministic Problem and Markov Jump Problem}

In what follows, we will show that to solve the Markov jump problem \eqref{eq:proq_markov}, it suffices to solve the deterministic problem \eqref{eq:proq_markov_equiv}. In particular, we will show that there exists an $r$-tuple $c^*=\big(c^*(1),\dots,c^*(r)\big)$ such that the optimal cost functions of the Markov jump problem \eqref{eq:proq_markov} and the deterministic problem \eqref{eq:proq_markov_equiv} are both characterized by $c^*$. Moreover, we will also show that there exists an optimal policy $\bar \mu\in \overline{\cal M}$ for the deterministic problem such that the corresponding policy $\mu\in\widehat{\cal M}$ is optimal for the Markov jump  problem. 

We first note that the functions $\overline{G}_\mu$ and $\overline{G}$ can be viewed as mappings from $\Re_+^{rn}$ to itself, and that the $[(\theta-1)n+1]$th to the $(\theta n)$th elements of the vectors $\overline{G}_\mu(c)$ and $\overline{G}(c)$ are $\big(\overline{G}_{\mu}(c)\big)(\theta)$ and $\big(\overline{G}(c)\big)(\theta)$, respectively. With these interpretations in mind, we are ready to state our result that connects the Markov jump problem \eqref{eq:proq_markov} and the deterministic problem \eqref{eq:proq_markov_equiv}, and establishes an equivalence relation between these two problems.

\begin{proposition}[Deterministic Equivalence]\label{prop:markov_ce}
    \begin{itemize}
        \item[(a)] There exists an $r$-tuple $c^*=\big(c^*(1),\dots,c^*(r)\big)$, $c^*(i)\in \re_+^n$ for $i=1,\dots,r$, such that the optimal cost functions $J^*$ and $\overline{J}^*$ of the Markov jump problem \eqref{eq:proq_markov} and the deterministic problem \eqref{eq:proq_markov_equiv}, respectively, satisfy 
        \begin{align*}
            &J^*(x,\theta)=c^*(\theta)'x,\quad\hbox{ for all } x\in X,\ \theta\in \Theta,\\
            &\overline{J}^*(\bar x)=\sum_{\theta=1}^rc^*(\theta)x^\theta, \quad\hbox{ for all }\bar x=(x^1,\dots,x^r)\in \overline{X}.
        \end{align*}
        Moreover, $c^*$ is the unique nonnegative $r$-tuple that satisfies 
        $$c^*(\theta)=\big(\overline{G}(c^*)\big)(\theta),\qquad \hbox{for all }\theta\in \Theta,$$ or equivalently,
        $c^*=\overline{G}(c^*).$
        \item[(b)] A policy $\mu\in \widehat{\cal M}$ is optimal for the Markov jump problem \eqref{eq:proq_markov} if and only if the corresponding policy $\bar\mu\in \overline{\cal M}$ is optimal for the deterministic problem  \eqref{eq:proq_markov_equiv}. Moreover, there exists at least one optimal policy for the Markov jump problem.
    \end{itemize}
\end{proposition}

\begin{proof}
By Prop.~\ref{prop:markov_equiv_prop}, Assumption~\ref{asm:dete_well} holds for the deterministic problem \eqref{eq:proq_markov_equiv}. As a consequence, we can use similar arguments to the ones used in the proof of Prop.~\ref{prop:stoch_ce} to establish both parts of the proposition.
\end{proof}

Regarding algorithms to compute the optimal vectors $c^*(1),\ldots,c^*(r)$ and a corresponding optimal policy, we can apply the VI, PI, and mathematical programming methods of Section \ref{sec:algo_dete} to  the deterministic problem  \eqref{eq:proq_markov_equiv}. This is straightforward, except that for the Markov jump problem, the state dimension is increased from $n$ (the dimension of the original state space $X$) to $rn$ (the dimension of the deterministic certainty equivalent state space). 

The complexity of the controller is also accordingly increased.  As an illustration, for positive linear Markov jump problems such as the one of Example 6.1, the result of the computation is a set of $m\times n$ linear gain matrices $L_1,\ldots,L_r$, one for each of the parameter values $1,\ldots,r$. The total number of optimal controller parameters is $mnr$, an increase of a factor of $r$ over the corresponding problem without jump parameters.

It is worth noting the distinction from the classical linear quadratic control problem with Markov jump parameters. In that problem, the DP algorithm yields a family of cost functions, one for each mode, and the corresponding Bellman's equations take the form of coupled matrix relations. These coupled matrix equations are \emph{not identical} to their counterparts without jump parameters; see \cite[Section~4.3]{costa2005discrete}. By contrast, our equivalence principle enlarges the state space in a way that preserves the semilinear structure. As a result, the deterministic reformulation leads to a single Bellman's equation that is \emph{identical} to those studied in Sections~\ref{sec:dete_formulation}–\ref{sec:algo_dete}. In this respect, our approach and results are closer to the filtering problem for Markov jump systems, where, under suitable conditions, an augmented deterministic formulation produces a matrix equation of the same form as in the classical case without jump parameters; see \cite[Section~5.4]{costa2005discrete}.

\section{Concluding Remarks}

In this paper, we have considered a broad class of infinite horizon DP problems characterized by partially linear structures and positivity properties in their state equations and cost functions. Our analysis encompasses both deterministic and stochastic formulations, including problems with Markov jump parameters. We have studied the existence and uniqueness of solutions to Bellman's equation, as well as the properties of optimal policies. Additionally, we have established the convergence of VI and PI algorithms, along with their variants. For problems involving stochastic Markov jump parameters, we have shown that a form of certainty equivalence holds, enabling the use of algorithms designed for deterministic problems.

Our semilinear DP theory integrates and extends several DP concepts. First, it expands the set of favorably structured DP problems, with a focus on the underlying idea of a structured pair of cost functions $\widehat{\cal J}$ and policies $\widehat{\cal M}$. Second, it generalizes the positive linear models developed in \cite{li2024exact} by providing a unifying framework for their analysis and by incorporating stochastic parameters. Finally, it formalizes the semilinear model, thereby unifying classes of problems that arise in a variety of application areas, such as engineering, biology, and economics.


\appendix

\section{Infinite Horizon Dynamic Programming with Nonnegative Cost per Stage}\label{app:a}
In this appendix we provide a brief account of the results  that we will use from infinite horizon DP theory for nonnegative cost problems. We consider discrete-time stochastic optimal control problems involving the system
\begin{equation}
\label{eq:dynamics}
    x_{k+1}=f(x_k,u_k,w_k),\quad k=0,\,1,\,\dots,
\end{equation}
where $x_k$ and $u_k$ are state and control at stage $k$, belonging to state and control spaces $X$ and $U$, respectively. The random disturbance $w_k$ takes values from a countable set $W$, with probability distribution that depends on the state-control pair $(x_k,u_k)$. The function $f$ maps $X\times U\times W$ to $X$. The control $u_k$ must be chosen from a nonempty constraint set $U(x_k)\subset U$ that may depend on $x_k$. The cost for the $k$th stage, denoted by $g(x_k,u_k,w_k)$, is assumed nonnegative:
\begin{equation}
\label{eq:cost_nonneg}
    g(x_k,u_k,w_k)\geq0,\qquad x_k\in X,\,u_k\in U(x_k),\,w_k\in W.
\end{equation}
We are interested in feedback policies of the form $\pi=\{\mu_0,\mu_1,\dots\}$, where $\mu_k$ is a function mapping a state $x\in X$ into the control $\mu_k(x)\in U(x)$. The set of all policies is denoted by $\Pi$. Policies of the form $\pi=\{\mu,\mu,\dots\}$ are called stationary, and when confusion cannot arise, will be denoted by $\mu$.

Given an initial state $x_0$, a policy $\pi=\{\mu_0,\mu_1,\dots\}$ when applied to system \eqref{eq:dynamics}, generates a random sequence of state control pairs $\big(x_k,\mu_k(x_k)\big),$ $k=0,1,\dots$, with cost 
$$J_{\pi}(x_0)=\lim_{N\to\infty}\mathop{E}_{\substack{w_k \\k=0,\dots,N-1}}\Bigg\{\sum_{k=0}^{N-1}\alpha^kg\big(x_k,\mu_k(x_k),w_k\big)\Bigg\},\qquad  x_0\in X.$$
We view $J_\pi$ as a nonnegative function over $X$, and we refer to it as the cost function of $\pi$. For a stationary policy $\mu$, the corresponding cost function is denoted by $J_\mu$. The optimal cost function is defined by
$$J^*(x)=\inf_{\pi\in\Pi}J_{\pi}(x),\qquad x\in X,$$
and a policy $\pi^*$ is said to be optimal if it attains the minimum of $J_\pi(x)$ for every $x\in X$, i.e., 
$$J_{\pi^*}(x)=\inf_{\pi\in \Pi}J_\pi(x)=J^*(x),\qquad x\in X.$$

The VI algorithm starts with some nonnegative function $J_0$, and generates a sequence of functions $\{J_k\}$ according to 
\begin{equation}
    \label{eq:vi}
    J_{k+1}(x)=\inf_{u\in U(x)}E\Big\{g(x,u,w)+\alpha J_k\big(f(x,u,w)\big)\Big\},\qquad x\in X.
\end{equation}
The PI algorithm starts with a stationary policy $\mu^0$, and generates a sequence of stationary policies $\{\mu^k\}$ via a sequence of policy evaluations to obtain $J_{\mu^k}$ from the equation 
\begin{equation}
    \label{eq:policy_ev}
    J_{\mu^k}(x)=E\Big\{g\big(x,\mu^k(x),w\big)+\alpha J_{\mu^k}\big(f\big(x,\mu^k(x),w\big)\big)\Big\},\qquad x\in X,
\end{equation}
interleaved with policy improvements to obtain $\mu^{k+1}$ from $J_{\mu^k}$ via 
\begin{equation}
    \label{eq:policy_im}
    \mu^{k+1}(x)\in \arg\min_{u\in U(x)}E\Big\{g(x,u,w)+\alpha J_{\mu^k}\big(f(x,u,w)\big)\Big\},\qquad x\in X,
\end{equation}
where we have assumed that the  minimum above is attained for every $x\in X$.

The following proposition provides the results that we will use in this paper.

\begin{proposition}\label{prop:classical}
Let the cost nonnegativity condition \eqref{eq:cost_nonneg} hold.
\begin{itemize}
    \item[(a)] $J^*$ satisfies Bellman's equation $$J^*(x)=\inf_{u\in U(x)}E\Big\{g(x,u,w)+\alpha J^*\big(f(x,u,w)\big)\Big\},\qquad \hbox{for all }x\in X,$$
and if a nonnegative function $\hat{J}$ satisfies
    $$\inf_{u\in U(x)}E\Big\{g(x,u,w)+\alpha \hat{J}\big(f(x,u,w)\big)\Big\}\leq \hat{J}(x),\qquad \hbox{for all }x\in X,$$
    then $J^*\leq \hat{J}$.
    \item[(b)] For all stationary policies $\mu$ we have 
    $$J_{\mu}(x)=E\Big\{g\big(x,\mu(x),w\big)+\alpha J_{\mu}\big(f\big(x,\mu(x),w\big)\big)\Big\},\qquad \hbox{for all }x\in X.$$
    Moreover, if a nonnegative function $\hat{J}$ satisfies
    \begin{equation}
    \label{eq:classic_mu_bound}
        E\Big\{g\big(x,\mu(x),w\big)+\alpha\hat{J}\big(f\big(x,\mu(x),w\big)\big)\Big\}\leq \hat{J}(x),\qquad \hbox{for all }x\in X,
    \end{equation}
    then $J_\mu\leq \hat{J}$.
    \item[(c)] A stationary policy $\mu^*$ is optimal if and only if 
    $$\mu^*(x)\in\arg\min_{u\in U(x)}E\Big\{g(x,u,w)+\alpha J^*\big(f(x,u,w)\big)\Big\},\qquad \hbox{for all }x\in X.$$
    \item[(d)] Let $\mu$ and $\tilde\mu$ be two stationary policies such that
    $$\tilde{\mu}(x)\in\arg\min_{u\in U(x)}E\Big\{g(x,u,w)+\alpha J_\mu\big(f(x,u,w)\big)\Big\},\qquad \hbox{for all }x\in X.$$
    Then $J_{\tilde\mu}\leq J_\mu$.
    \item[(e)] If a function $J$ satisfies 
    $$J^*\leq J\leq sJ^*$$
    for some scalar $s>1$, then the sequence $\{J_k\}$ generated by VI with $J_0=J$ converges to $J^*$, i.e., $J_k\to J^*$.
\end{itemize}
\end{proposition}
 Parts (a)-(d) of Prop.~\ref{prop:classical} are well known and can be found in several sources; see the books \cite{bertsekas1978stochastic}, \cite{bertsekas2012dynamic}, and the references cited there. Prop.~\ref{prop:classical}(e) is a less known result, which was first formulated and proved in the paper by Yu and Bertsekas \cite[Theorem~5.1]{yu2015mixed}, in a form that also addressed the associated measurability issues for stochastic optimal control problems. This result will be used to assert the uniqueness of solution of Bellman's equation and the convergence of VI within our semilinear context. A simpler form of this result, which applies to deterministic problems, or problems without measurability restrictions, such as the ones of the present paper, is given in the abstract DP book \cite[Prop.~4.4.6, Ch.~4]{bertsekas2022abstract}.

Another important result that we need in our analysis is the Perron-Frobenius theorem. The proof of the following version can be found in \cite[Theorem 3, Section 6.2]{luenberger1979introduction}; see also \cite[Prop.~6.6, Ch.~2]{bertsekas1989parallel}.
\begin{proposition}\label{prop:f_p}
    Let $A$ be an $n\times n$ matrix such that $A\geq 0$. Then there exists some scalar $\lambda\geq0$ and some nonzero vector $v\in \Re_+^n$ such that $Av=\lambda x$, while for all other eigenvalues $\rho$ of $A$, we have $|\rho|\leq \lambda$.
\end{proposition}

\section{Proof of Prop.~\ref{prop:markov_equiv_prop}}\label{app:proof}
We give the proof of Prop.~\ref{prop:markov_equiv_prop} in two steps. First, we  show that the deterministic problem~\eqref{eq:proq_markov_equiv} preserves the semilinear structure: under any policy $\bar \mu\in\overline{\cal M}$, both the state equation and cost per stage of problem \eqref{eq:proq_markov_equiv} becomes linear. Then, we  show that Assumption~\ref{asm:markov_well}, which applies to the Markov jump problem \eqref{eq:proq_markov}, implies that the deterministic problem \eqref{eq:proq_markov_equiv} satisfies Assumption~\ref{asm:dete_well}, with appropriate changes in notation.

\begin{proof}
By the definition of $\Bar{f}^\theta$ [cf. Eq.~\eqref{eq:mean_nl_dy_markov}], for every $\bar \mu\in \overline{\cal M}$ and $\bar x=(x^1,\dots,x^r)\in\overline{X}$, we have
$$\Bar{f}^\theta\big(\bar x,\bar\mu(\bar x)\big)=\sum_{i=1}^rp_{i\theta}A_\mu^{i\theta}x^i,\quad \theta=1,\dots,r,$$
where $\mu\in \widehat{\cal M}$ corresponds to $\bar \mu$ as per Eq.~\eqref{eq:markov_extend_mu} and the equality follows from Eq.~\eqref{eq:dy_markov_dete}. From Eq.~\eqref{eq:markov_extend_dy},
\begin{equation}
    \label{eq:markov_extended_full}
    \Bar{f}\big(\bar x,\bar\mu(\bar x)\big)=\begin{bmatrix}
        \sum_{i=1}^rp_{i1}A_\mu^{i1}x^i\\
        \sum_{i=1}^rp_{i2}A_\mu^{i2}x^i\\
        \vdots\\
        \sum_{i=1}^rp_{ir}A_\mu^{ir}x^i
    \end{bmatrix}=\begin{bmatrix}
p_{11}A_\mu^{11} & p_{21}A_\mu^{21} & \cdots & p_{r1}A_\mu^{r1}\\
p_{12}A_\mu^{12} & p_{22}A_\mu^{22} & \cdots & p_{r2}A_\mu^{r2}\\
\vdots  & \vdots  & \ddots & \vdots  \\
p_{1r}A_\mu^{1r} & p_{2r}A_\mu^{2r} & \cdots & p_{rr}A_\mu^{rr}
\end{bmatrix}\begin{bmatrix}
        x^1\\
        x^2\\
        \vdots\\
        x^r
    \end{bmatrix}.
\end{equation}
By the definition of $\Bar{A}_\mu$ in Eq.~\eqref{eq:a_mu_markov_dete}, Eq.~\eqref{eq:markov_extended_full} can be compactly written as $\Bar{f}\big(\bar x,\mu(\bar x)\big)=\Bar{A}_\mu \bar x$. Similarly, for $\bar \mu\in\overline{\cal M}$, we have that $\Bar{g}\big(\bar x,\bar \mu(\bar x)\big)=\Bar{q}_\mu'\bar x,$ in view of the definition of $\Bar{q}_\mu$, cf., Eq.~\eqref{eq:a_mu_markov_dete}. Hence, under any policy $\bar\mu\in \overline{\cal M}$, the state equation and the cost per stage of problem \eqref{eq:proq_markov_equiv} are both linear in $\bar x$, confirming the semilinear structure.

Let us now verify that Assumption~\ref{asm:dete_well} holds for problem \eqref{eq:proq_markov_equiv}. To show that part (a) holds, let us consider the expression of $\inf_{\bar u\in \overline{U}(\bar x)}\big\{\Bar{g}(\bar x,\bar u)+\alpha c'\Bar{f}(\bar x,\bar u)\big\}$. Straightforward calculation yields: 
\begin{equation}
    \label{eq:markov_asm_well_c_calculation}
    \begin{aligned}
    &\inf_{\bar u\in \overline{U}(\bar x)}\big\{\Bar{g}(\bar x,\bar u)+\alpha c'\Bar{f}(\bar x,\bar u)\big\}\\
    =&\inf_{(u^1,\dots,u^\theta)\in \overline{U}(\bar x)}\Bigg\{\sum_{\theta=1}^rE_{w}\Big\{g(x^\theta,\theta,u^\theta,w)\,\big|\,\theta\Big\}+\alpha \sum_{w=1}^rc'(w)\Bar{f}^{w}(\bar x,\bar u)\Bigg\}\\
    =&\inf_{(u^1,\dots,u^\theta)\in \overline{U}(\bar x)}\Bigg\{\sum_{\theta=1}^rE_{w}\Big\{g(x^\theta,\theta,u^\theta,w)\,\big|\,\theta\Big\}+\alpha \sum_{w=1}^rc'(w)\sum_{\theta=1}^rp_{\theta w}f(x^\theta,\theta,u^\theta,w)\Bigg\}\\
    =&\inf_{(u^1,\dots,u^\theta)\in \overline{U}(\bar x)}\Bigg\{\sum_{\theta=1}^rE_{w}\Big\{g(x^\theta,\theta,u^\theta,w)\,\big|\,\theta\Big\}+\alpha \sum_{\theta=1}^r\sum_{w=1}^rp_{\theta w}c'(w)f(x^\theta,\theta,u^\theta,w)\Bigg\}\\
    =&\inf_{(u^1,\dots,u^\theta)\in \overline{U}(\bar x)}\Bigg\{\sum_{\theta=1}^rE_{w}\Big\{g(x^\theta,\theta,u^\theta,w)\,\big|\,\theta\Big\}+\alpha \sum_{\theta=1}^rE_{w}\Big\{c'(w)f(x^\theta,\theta,u^\theta,w)\,\big|\,\theta\Big\}\Bigg\}\\
    =&\inf_{(u^1,\dots,u^\theta)\in \overline{U}(\bar x)}\sum_{\theta=1}^rE_{w}\Big\{g(x^\theta,\theta,u^\theta,w)+\alpha c'(w)f(x^\theta,\theta,u^\theta,w)\,\big|\,\theta\Big\}\\
    =&\inf_{(u^1,\dots,u^\theta)\in U(x^1,1)\times\dots \times U(x^r,r)}\sum_{\theta=1}^rE_{w}\Big\{g(x^\theta,\theta,u^\theta,w)+\alpha c'(w)f(x^\theta,\theta,u^\theta,w)\,\big|\,\theta\Big\}\\
    =&\sum_{\theta=1}^r\inf_{u^\theta\in U(x^\theta,\theta)}E_{w}\Big\{g(x^\theta,\theta,u^\theta,w)+\alpha c'(w)f(x^\theta,\theta,u^\theta,w)\,\big|\,\theta\Big\}.
    \end{aligned}
\end{equation}
The first equality above holds by the definition of elements in $\overline{U}(\bar x)$ [cf., Eq.~\eqref{eq:markov_u_bar_def}], $\Bar{g}$ [cf., Eq.~\eqref{eq:markov_extend_cost}], and $\Bar{f}$ [cf., Eq.~\eqref{eq:markov_extend_dy}]. The second equality holds by the definition of $\Bar{f}^\theta$ [cf., Eq.~\eqref{eq:mean_nl_dy_markov}]. In the third equality, we exchange the order of the summation of $\theta$ and $w$ for the terms $p_{\theta w}c'(w)f(x^\theta,\theta,u^\theta,w)$, and in the fourth equality, we write these sums as conditional expectations. In the second to last equality, we use the definition of $\overline{U}(\bar x)$ given in Eq.~\eqref{eq:markov_u_bar_def}. The last equality holds because minimizing the term $E_{w}\Big\{g(x^\theta,\theta,u^\theta,w)+\alpha c'(w)f(x^\theta,\theta,u^\theta,w)\,\big|\,\theta\Big\}$ over $u^\theta$ is completely independent of other $\hat{\theta}\neq \theta$. 

By the definition of $\overline{G}$ [cf. Eq.~\eqref{eq:markov_g_bar}], we have that
\begin{equation}
    \label{eq:markov_equiv_g}
    \big(\overline{G}(c)\big)(\theta)'x^\theta =\inf_{u^\theta\in U(x^\theta,\theta)}E_{w}\big\{g(x^\theta,\theta,u^\theta,w)+\alpha c'(w)f(x^\theta,\theta,x^\theta,w)\,|\,\theta\big\}
\end{equation}
for $\theta=1,\dots,r$. Combining Eqs.~\eqref{eq:markov_asm_well_c_calculation} with \eqref{eq:markov_equiv_g} yields
\begin{equation}
\label{eq:markov_equiv_bellman_op}
    \overline{G}(c)'\bar x=\sum_{\theta=1}^r\big(\overline{G}(c)\big)(\theta)'x^\theta=\inf_{\bar u\in \overline{U}(\bar x)}\big\{\Bar{g}(\bar x,\bar u)+\alpha c'\Bar{f}(\bar x,\bar u)\big\}.
\end{equation}
In other words, we have that 
\begin{equation}
    \label{eq:markov_equiv_asm_a}
    \hat{c}'\bar x=\inf_{\bar u\in \overline{U}(\bar x)}\big\{\Bar{g}(\bar x,\bar u)+\alpha c'\Bar{f}(\bar x,\bar u)\big\},
\end{equation}
where $\hat{c}=\overline{G}(c)$, and $\hat{c}$ depends continuously on $c$ since $\overline{G}$ is continuous by Assumption~\ref{asm:markov_well}(a).

Next, we show that minimum in Eq.~\eqref{eq:markov_equiv_asm_a} is attained by some policy $\bar\mu\in \overline{\cal M}$. By Assumption~\ref{asm:markov_well}(a), for every $c\in \mathcal{C}$ and $\theta\in \Theta$, there exists $\mu^\theta\in \widehat{\cal M}$, such that
\begin{equation}
    \label{eq:markov_asm_well_c}
    \big(\overline{G}_{\mu^\theta}(c)\big)(\theta)'x^\theta=\inf_{u^\theta\in U(x^\theta,\theta)}E_{w}\big\{g(x^\theta,\theta,u^\theta,w)+\alpha c'(w)f(x^\theta,\theta,x^\theta,w)\,|\,\theta\big\}
\end{equation}
for all $\bar x=(x^1,\dots,x^r)\in \Re_+^{rn}$. In view of Assumption~\ref{asm:markov_well}(e), there exists a policy $\tilde\mu\in\widehat{\cal M}$ such that $\tilde\mu(x^\theta,\theta)=\mu^\theta(x^\theta,\theta)$ for all $x^\theta\in X,\,\theta\in \Theta$. Hence, 
\begin{equation}
    \label{eq:markov_equiv_common_mu}
    \big(\overline{G}_{\tilde\mu}(c)\big)(\theta)=\big(\overline{G}_{\mu^\theta}(c)\big)(\theta),\quad\text{for all }\theta\in \Theta.
\end{equation}
Combining Eqs.~\eqref{eq:markov_asm_well_c_calculation}, \eqref{eq:markov_asm_well_c}, and \eqref{eq:markov_equiv_common_mu}, we obtain
\begin{equation}
\label{eq:op_policy_markov}
    \overline{G}_{\tilde\mu}(c)'\bar x=\inf_{\bar u\in \overline{U}(\bar x)}\big\{\Bar{g}(\bar x,\bar u)+\alpha c'\Bar{f}(\bar x,\bar u)\big\}.
\end{equation}
Let $\bar\mu\in \overline{\cal M}$ be a policy that corresponds to $\tilde \mu\in \widehat{\cal M}$ as per Eq.~\eqref{eq:markov_extend_mu}. Then the minimum is attained at $\bar \mu$.

Part (b) holds directly for problem \eqref{eq:proq_markov_equiv} in view of Assumption~\ref{asm:markov_well}(b). To show that Assumption~\ref{asm:dete_well}(c) holds for problem~\eqref{eq:proq_markov_equiv}, let $v=(v^1,\dots,v^r)$ be a vector in $\Re^{rn}_+$. By Assumption~\ref{asm:markov_well}(c), there exist vectors $x^\theta\in X$ and scalars $s^\theta\geq 0$, $\theta=1,\dots,r$, such that $s^\theta x^\theta= v^\theta$ for all $\theta$. As a result, we have that $s^\theta x^\theta/s\in X$, where $s=\max\{s^1,\dots,s^r\}$, and $v=s\cdot[(s^1x^1,\dots,s^rx^r)/s]$, and $(s^1x^1,\dots,s^rx^r)/s\in \overline{X}$.

We now show that the deterministic problem \eqref{eq:proq_markov_equiv} also satisfies Assumption~\ref{asm:dete_well}(d). By a principle of optimality argument (see, e.g., \cite[Section~1.3]{bertsekas2017dynamic}), the problem of minimizing the cost 
$$\mathop{E}_{\substack{w_k \\k=0,\dots,N-1}}\Bigg\{\sum_{k=0}^{N-1}\alpha^kg(x_k,\theta_k,u_k,w_{k})\Bigg\},$$
starting from $(x_0,\theta_0)$ has the optimal value $\big(\overline{G}^N(0)\big)(\theta_0)'x_0$, where $0\in \re_+^{rn}$ stands for the $rn$-dimensional zero vector. The value $\big(\overline{G}^N(0)\big)(\theta_0)'x_0$ is positive for all nonzero $x_0$ and all parameters $\theta_0$. Let us now consider $\overline{G}^N(0)'\bar x$, where $\bar x=(x^1,\dots,x^r)$ is nonzero. We have that
$$\overline{G}^N(0)'\bar x=\sum_{\theta=1}^r\big(\overline{G}^N(0)\big)(\theta)'x^\theta>0,$$
where the inequality follows from the fact that $x^\theta$ is nonzero for some $\theta$. Again using a principle of optimality argument and by Eq.~\eqref{eq:markov_equiv_bellman_op}, $\overline{G}^N(0)'\bar x$ is the optimal cost of the problem of minimizing the cost $\sum_{k=0}^{N-1}\alpha^k\Bar{g}(\bar x_k,\bar u_k)$ with $\bar x_0=\bar x$. As a result, Assumption~\ref{asm:dete_well}(d) holds for problem \eqref{eq:proq_markov_equiv}. The proof is complete.
\end{proof}

\bibliographystyle{alpha}
\bibliography{ref}

\end{document}